\documentclass[11pt]{article}     % onecolumn (ditto)
\usepackage{graphicx}
\usepackage{amssymb}
\usepackage{amsthm}
\usepackage[T1]{fontenc}
\usepackage[utf8]{inputenc}
\usepackage{newtxtext,newtxmath}
\usepackage{hyperref}
  \hypersetup{
    colorlinks,
    linkcolor=blue,%{red!50!black},
    citecolor=blue,%{blue!50!black},
    urlcolor=blue,%{blue!80!black},
    final
  }
\usepackage{textcomp} 
\usepackage{doi}
\usepackage{amsmath,amssymb,mathtools}
\usepackage{bm}
\usepackage{float}
\usepackage{siunitx,booktabs}
\usepackage{subcaption}
\usepackage[square,sort,comma,numbers]{natbib}
\usepackage{algorithm}
\usepackage{algorithmic}

\newtheorem{theorem}{Theorem}[section]

\newtheorem{proposition}[theorem]{Proposition}

\newtheorem{definition}{Definition}[section]

\newtheorem{remark}{Remark}

\usepackage[top=1in, bottom=1in, left=1in, right=1in]{geometry}
\title{Difference Potentials Method for Models with Dynamic Boundary Conditions and Bulk-Surface Problems}
\author{Yekaterina Epshteyn \thanks{Department of Mathematics, The University of Utah, 155 S 1400 E Room 233, Salt Lake City, Utah 84112, USA, Email: epshteyn@math.utah.edu} \and Qing Xia \thanks{Department of Mathematics, The University of Utah, 155 S 1400 E Room 233, Salt Lake City, Utah 84112, USA, Email: xia@math.utah.edu}}

\providecommand{\keywords}[1]{\noindent\textbf{Keywords} #1}

\providecommand{\subclass}[1]{\noindent\textbf{AMS Subject Classification} #1}

\begin{document}

\maketitle

%%%%%%%%%%%%%%%%%%%%%%%%%%%%%%%%%%%%%%%%%%%%%%%%%%%%%%%%%%%%%%%%%%%%%%%%%%%%%%%%%%%%%%%%%%%%%%%%%%%%%%%%
% Abstract
%%%%%%%%%%%%%%%%%%%%%%%%%%%%%%%%%%%%%%%%%%%%%%%%%%%%%%%%%%%%%%%%%%%%%%%%%%%%%%%%%%%%%%%%%%%%%%%%%%%%%%%%
\begin{abstract}
In this work, we consider parabolic models with dynamic boundary
conditions and parabolic bulk-surface problems in 3D. Such partial
differential equations based models describe phenomena that happen
both on the surface and in the bulk/domain. These problems may appear in many
applications, ranging from cell dynamics in biology, to grain growth
models in polycrystalline materials. Using Difference Potentials
framework, we develop novel numerical algorithms for the approximation
of the problems. The constructed algorithms efficiently and
accurately handle the coupling of the models in the bulk and on the surface, approximate 3D irregular
geometry in the bulk by the use of only Cartesian meshes, employ Fast
Poisson Solvers, and utilize spectral approximation on the
surface. Several numerical tests are given to illustrate the robustness of the developed numerical algorithms.
\end{abstract}


\subclass{65M06, 65M12, 65M70, 35K10}

%\tableofcontents

\section{Introduction} \label{sec:intro}

The parabolic models with dynamic boundary conditions and parabolic
bulk-surface models can be found in a variety of applications in fluid
dynamics, materials science and biological applications, see for example,
\cite{MR2455779,MR2413263,Novak_2007,Cusseddu_2018,Elliott_2017,MR3129523,liu2017energetic,MR3729587,MR3606421,MR3423226,MADZVAMUSE20169,10.1007/978-3-319-71431-8_8,MR3361677}. In
many of these problems,  partial differential equations (PDE) based models are
used to capture dynamic phenomena that occur
on the surface of the domain and in the bulk/domain. For
instance, cell polarizations can be modeled by the switches of
Rho GTPases between the active forms on the membrane (surface) and inactive
forms in the cytosol (bulk) \cite{Cusseddu_2018}. Another example
is the modeling of the receptor-ligand dynamics, 
\cite{Elliott_2017}, to name a few examples here.  

In the current literature, there are only few numerical methods
developed for such problems, and most of the methods are
finite-element-based. For instance, a novel finite element scheme is
proposed and analyzed for 3D elliptic bulk-surface problems in
\cite{Elliott_2012}, where polyhedral elements are constructed in the
bulk region, and the piecewise polynomial boundary faces serve as the
approximation of the surface. The method in \cite{Elliott_2012}
employs two finite-element spaces, one in the bulk, and one on the
surface. See also the review paper
\cite{MR3038698} on the finite element methods for PDEs on curved surfaces and the references therein. Also, space and time discretizations of 2D heat equations with dynamic boundary conditions are studied in \cite{Kov_cs_2016}, in a weak formulation that fits into the standard variational framework of parabolic problems.
A flexible unfitted finite element method (cut-FEM) is proposed for 3D
elliptic bulk-surface problems in \cite{Burman_2015}. The developed
cut-FEM utilizes the same finite element space defined on a structured
background mesh to solve the PDEs in the bulk region and on the
surface. Another space-time cut-FEM approach, with continuous linear
elements in space and discontinuous piecewise linear elements in time,
is designed for 2D parabolic bulk-surface problems on time-dependent
domains in \cite{Hansbo_2016}. Furthermore, a hybrid finite-volume-finite-element method is developed for 3D bulk-surface models in \cite{MR3717149}. The hybrid method employs a monotone nonlinear finite volume method in the bulk, and the trace finite element method \cite{MR2570076,MR3806652} is used to solve equations on the reconstructed polygonal approximation of the surface.

In this work, we develop novel numerical algorithms for 3D models with
dynamic boundary conditions and bulk-surface coupling, within the
framework of Difference Potentials method (DPM). The constructed
numerical schemes efficiently and
accurately handle the coupling of the models in the bulk and on the surface, approximate 3D irregular
geometry in the bulk by the use of only Cartesian grids, employ Fast
Poisson Solvers, and apply spectral approximation on the
surface.

The paper is organized as follows. In Section~\ref{sec:model}, we discuss the two distinct yet
related model problems that are considered in the current work, the parabolic model with dynamic boundary
condition and parabolic bulk-surface problem in 3D. Next,
in Section~\ref{sec:dpm_sd}, we develop numerical methods
based on Difference Potentials for these problems,  and give the
main steps of the constructed numerical algorithms. Lastly, in
Section~\ref{sec:numerics}, we present the extensive
numerical results (convergence, 3D views of the solutions,
  etc.) that show the robustness of the developed algorithms.

%%%%%%%%%%%%%%%%%%%%%%%%%%%%%%%%%%%%%%%%%%%%%%%%%%%%%%%%%%%%%%%%%%%%%%%%%%%%%%%%%%%%%%%%%%%%%%%%%%%%%%%%%%%
%%%%%%%%%%%%%%%%%%%%%%%%%%%%%%%%%%%%%%%%%%%%%%%%%%%%%%%%%%%%%%%%%%%%%%%%%%%%%%%%%%%%%%%%%%%%%%%%%%%%%%%%%%%
\section{The Model with Dynamic Boundary Condition and Bulk-Surface Problem} \label{sec:model}
In this work, we consider the following two models in 3D:\\
{\it Heat equation with dynamic boundary condition on the surface (see
  related examples in \cite{V_zquez_2011,Kov_cs_2016}),}
\begin{align}
    u_t-\Delta u=f,&\quad (x, y, z, t)\in \Omega\times\mathbb{R}^+,\label{eqn:dynamic-eq}\\
    u_t+u+ n\cdot\nabla u=\Delta_\Gamma u+g,&\quad (x, y, z, t)\in \Gamma\times\mathbb{R}^+,\label{eqn:dynamic-bc}\\
    u(x,y,z,0) = u_0(x, y, z),&\quad (x, y, z) \in\Omega\cup\Gamma.\label{eqn:dynamic-ic}
\end{align}
{\it The bulk-surface problem (see related examples in \cite{Hansbo_2016,Elliott_2017}),}
\begin{align}
    u_t-\Delta u=f,&\quad (x, y, z, t)\in \Omega\times\mathbb{R}^+,\label{eqn:bulk}\\
    -n\cdot \nabla u=h(u,v),&\quad (x, y, z, t)\in \Gamma\times\mathbb{R}^+,\label{eqn:coupling}\\
    v_t-\Delta_\Gamma v = g+h(u,v),&\quad (x, y, z, t)\in \Gamma\times\mathbb{R}^+,\label{eqn:surface}\\
    u(x, y, z, 0) = u_0(x, y, z),&\quad (x, y, z)\in\Omega, \label{eqn:bulk-ic}\\
    v(x,y,z,0) = v_0(x, y, z),&\quad (x, y, z) \in\Gamma.    \label{eqn:surface-ic}
\end{align}
In the above models,  $\Gamma$ is a smooth boundary/surface of a bounded domain/bulk $\Omega\subset \mathbb{R}^3$, $\Delta_\Gamma$ is the Laplace-Beltrami operator defined on $\Gamma$,  $n$ denotes the outward unit normal vector. The function $h(u,v)$ is the coupling relation between the bulk and the surface, and $g$ in (\ref{eqn:dynamic-bc}) or (\ref{eqn:surface}) is
the source function on the surface. The initial data for the model
(\ref{eqn:dynamic-eq})-(\ref{eqn:dynamic-bc}) is given by function
$u_0(x, y, z), (x, y, z) \in\Omega\cup\Gamma$ and the initial data in
(\ref{eqn:bulk})-(\ref{eqn:surface}) are given by functions  $u_0(x,
y, z),
(x, y, z)\in\Omega$ and    $v_0(x, y, z),  (x, y, z)\in\Gamma$.

\section{Algorithms Based on DPM}\label{sec:dpm_sd}

The current work is a continuation of the recent work in
\cite{AlEpshSt,MR3659255,MR3954435,MR3817808}. For the time being, we will consider the model
with dynamic boundary conditions and the bulk-surface problem in a
spherical domain, but the proposed methods can be extended to domains
with more general geometry in 3D (and the main ideas of the algorithms will stay
the same, see Remark \ref{remark:derivatives-shperical-harmonics} below). We employ a finite-difference scheme for the underlying
space discretization of the models in the bulk (\ref{eqn:dynamic-eq}) or
(\ref{eqn:bulk}), combined with the idea of Difference Potentials
Method (DPM) (\cite{Ryab} and very recent work \cite{Ryaben_kii_2006,Epsh,AlEpshSt,MR3659255,MR3954435}, etc.)  that
provides flexibility to handle irregular domains and nontrivial
boundary conditions (including, but not limited to, dynamic boundary
conditions like \eqref{eqn:dynamic-bc}, or surface equations like \eqref{eqn:surface}) accurately and efficiently.

%For applications of bulk-surface coupling in fluid dynamics or molecule biology, see \cite{MR3717149,Colli_2014,Cusseddu_2018,Egger_2018,Gross_2015,MacDonald_2016,Novak_2007}.

\subsection{The Numerical Algorithm Based on DPM}\label{sec:algorithm}
%%%%%%%%%%%%%%%%%%%%%%%%%%%%%%Feb 17 2019: started to work below
%%%%%%%%%%%%%%%%%%%%%%%%%%%%%%between Feb 17
%%%%%%%%%%%%%%%%%%%%%%%%%%%%%%marks %%%%%%%%%%%%%%%%%%%%%%
\par {\bf Discretization in the Bulk:}
\begin{figure}
    \centering
    \begin{subfigure}[b]{0.4\textwidth}
        \includegraphics[width=\textwidth]{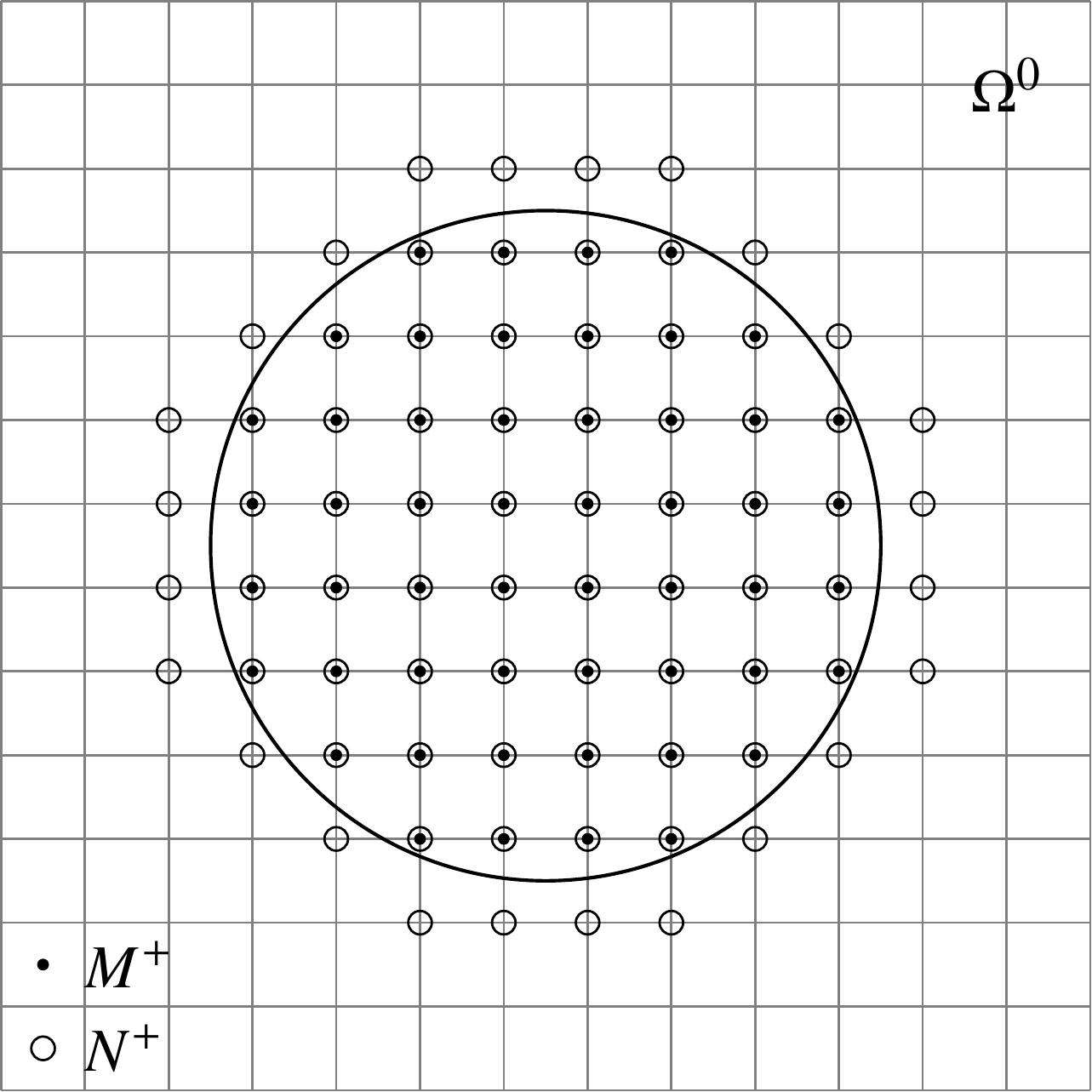}
        \caption{}
    \end{subfigure}
    ~
    \begin{subfigure}[b]{0.4\textwidth}
        \includegraphics[width=\textwidth]{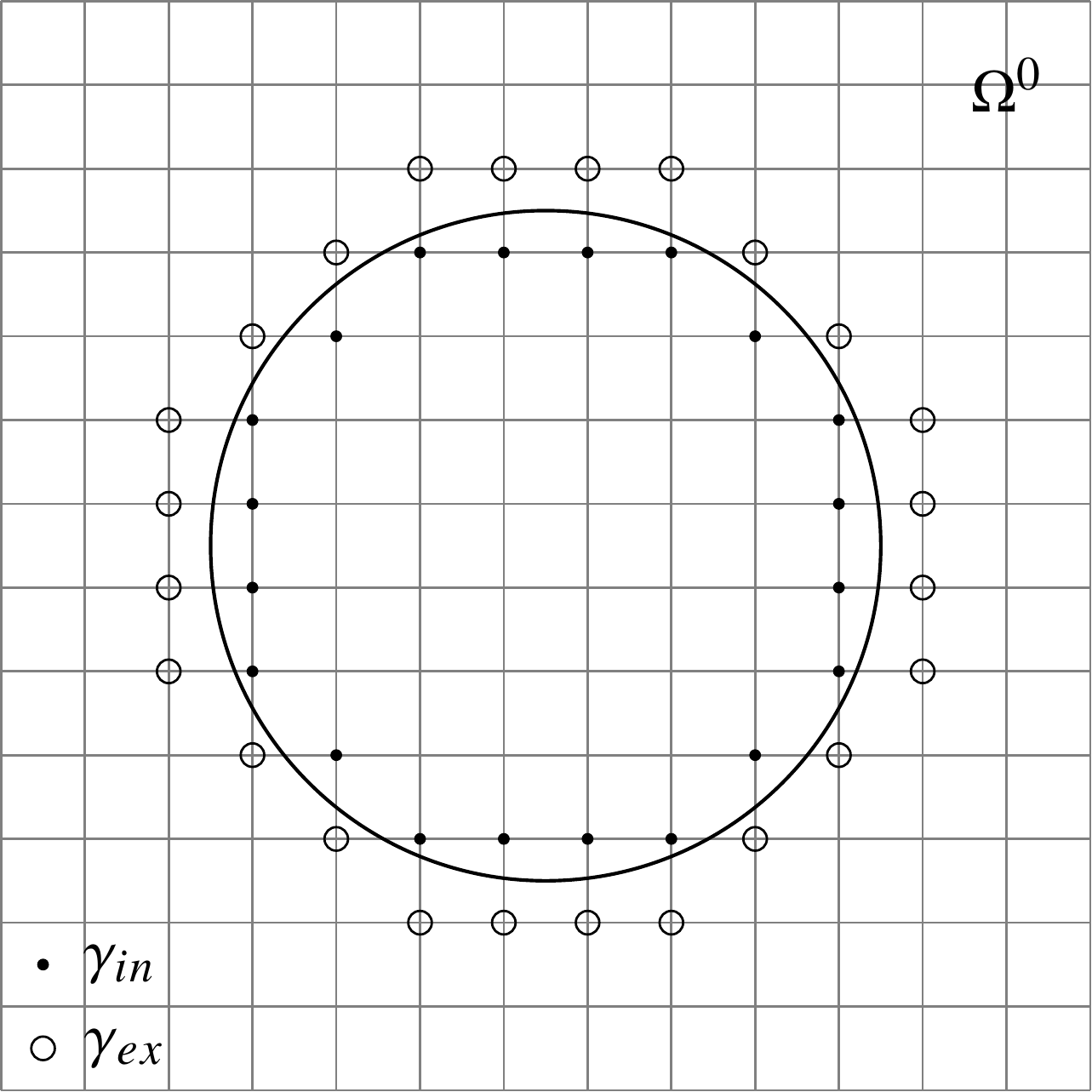}
        \caption{}
    \end{subfigure}
    \caption{Examples of point sets in the cross-sectional view: $M^+$ (solid dots)
as a subset of $N^+$ (open circles), where solid dots in open circles show the
overlap between $M^+$ and $N^+$ in the left figure; and  the discrete grid
boundary $\gamma$ as the union of $\gamma_{ex}$ (open circles) and $\gamma_{in}$
(solid dots) in the right figure. The auxiliary domain is denoted as $\Omega^0$
in both figures.}
    \label{fig:point-sets}
\end{figure}
\paragraph{Introduction of the Auxiliary Domain.} As a first step of
the numerical algorithm, we embed the original domain $\Omega$
into a computationally simple auxiliary domain
$\Omega^0\subset\mathbb{R}^3$, that we will select to be a cube in
this work. Next,
we introduce a  Cartesian mesh to discretize the auxiliary
domain $\Omega^0$, with mesh nodes $(x_j,y_k,z_l)=(x_0+j\Delta
x,y_0+k\Delta y,z_0+l\Delta z)$, ($j,k,l=0,1,2\dots,N$). Here,
$(x_0,y_0,z_0)$ is the left-bottom corner point of the cubical
auxiliary domain $\Omega^0$. For simplicity, we assume that the
Cartesian mesh is uniform, i.e., $h:=\Delta x=\Delta y=\Delta z$. To
discretize the PDE (\ref{eqn:dynamic-eq}) or
(\ref{eqn:bulk}) in the bulk, with a second order accuracy in space, we will consider the standard 7-point 
finite-difference stencil with a center placed at the point $(x_j,y_k,z_l)$:
\begin{align}\label{eqn:7-point-stencil}
\mathcal{N}_{j,k,l}^7 = \left\{(x_j,y_k,z_l),(x_{j\pm1},y_k,z_l),(x_j,y_{k\pm1},z_l),(x_j,y_k,z_{l\pm1})\right\}.
\end{align}
Next, we define the important point sets that we will use as a part of the
Difference Potentials framework (see Fig.~\ref{fig:point-sets}):\\
\begin{definition}\label{def:point_sets} 
Introduce the following point sets:\\
\begin{itemize}
    \item 

$M^0= \left\{(x_j,y_k,z_l)\mid(x_j,y_k,z_l)\in\Omega^0\right\}$
denotes the set of all mesh nodes $(x_j,y_k,z_l)$ that belong to the interior of the auxiliary domain $\Omega^0$;
    \item
      $M^+=M^0\cap\Omega=\left\{(x_j,y_k,z_l)\mid(x_j,y_k,z_l)\in\Omega\right\}$
      denotes the set of all mesh nodes $(x_j,y_k,z_l)$ that belong to the interior of the original domain $\Omega$;
    \item
    $M^-=M^0\backslash
    M^+=\{(x_j,y_k,z_l)\mid(x_j,y_k,z_l)\in\Omega^0\backslash\Omega\}$
    is the set of all mesh nodes $(x_j,y_k,z_l)$ that are inside of the auxiliary domain $\Omega^0$,  but belong to the exterior of the original domain $\Omega$;
    \item $N^+=\left\{\bigcup_{j,k,l}\mathcal{N}_{j,k,l}^{7}\mid(x_j,y_k,z_l)\in M^+\right\}$;
    \item $N^-=\left\{\bigcup_{j,k,l}\mathcal{N}_{j,k,l}^{7}\mid(x_j,y_k,z_l)\in M^-\right\}$;
    \item
      $N^0=\left\{\bigcup_{j,k,l}\mathcal{N}_{j,k,l}^{7}\mid(x_j,y_k,z_l)\in
        M^0\right\}$;\\ The point sets $N^\pm$ and $N^0$ are the sets
      of all mesh nodes covered by the stencil $\mathcal{N}^7_{j,k,l}$
      for every mesh node $(x_j,y_k,z_l)$ in $M^\pm$ and $M^0$ respectively;
    \item $\gamma=N^+\cap N^-$ defines a thin layer of mesh nodes that straddles the continuous boundary $\Gamma$ and is called the discrete grid boundary;
    \item $\gamma_{in}=M^+\cap\gamma$ and $\gamma_{ex}=M^-\cap\gamma$ are subsets of the discrete grid boundary that lie inside and outside of the spherical domain $\Omega$ respectively.
\end{itemize}
\end{definition}

\paragraph{Construction of the System of Discrete Equations for Models (\ref{eqn:dynamic-eq}) and (\ref{eqn:bulk}).} 
In this work, we will use the trapezoidal time stepping (Crank-Nicolson scheme) to
illustrate the approach based on Difference Potentials for the models with
dynamic boundary conditions and for the bulk-surface problems. In general, any
other stable
time marching scheme can be employed in a similar way. 

For the spatial discretization, we will employ the second-order finite-difference scheme using the 7-point stencil $\mathcal{N}_{j,k,l}^7$ as
defined above. Assume now, that $u^{i}_{j, k, l}$ denotes a discrete solution
computed at the time level $t^i$ at the mesh node $(x_j, y_k, z_l)$. Then, the
discrete system of equations  for (\ref{eqn:dynamic-eq}) and
(\ref{eqn:bulk}) obtained using trapezoidal time
approximation combined with the second-order central finite-difference
approximation in space is,
\begin{align}\label{fully-discrete-cn}
L_{h,\Delta
  t} u^{i+1}_{j,k,l}=F^{i+1}_{j, k, l},&\quad (x_j,y_k,z_l)\in M^+,
\end{align}
where, we introduced the discrete linear difference operator $L_{h,\Delta
  t}\equiv \Delta_h-\sigma I$ with $\sigma=2/\Delta t$, $\Delta_h$--the discrete Laplace operator defined on point set $M^+$, $I$--the identity matrix of the same size as $\Delta_h$, the right-hand side function
$F^{i+1}_{j, k, l}\equiv-(\Delta_h+\sigma I) u^i_{j,k,l}-f^{i+1}_{j,k,l}-f^{i}_{j,k,l}$, and $u^{i+1}_{j,k,l}\approx u(x_j,y_k,z_l,t^{i+1})$.
\paragraph{The Discrete Auxiliary Problem (AP).}
One of the important steps of DPM-based methods is the introduction of
the auxiliary problem (AP). The discrete APs play a key role in construction of the \textit{Particular Solution} and the
\textit{Difference Potentials} operators as a part of DPM-based algorithm proposed in this work.
\begin{definition}
At time $t^{i+1}$, given the grid function $q^{i+1}$ on $M^0$, the following difference equations~\eqref{eqn:discrete_ap}--\eqref{eqn:discrete_ap_boundary} are defined as the discrete Auxiliary Problem (AP):
\begin{align}
L_{h,\Delta t}w_{j,k,l}^{i+1}&=q^{i+1}_{j,k,l},\quad(x_j,y_k,z_l)\in M^0,\label{eqn:discrete_ap}\\
w^{i+1}_{j,k,l}&=0,\quad(x_j,y_k,z_l)\in N^0\backslash M^0.\label{eqn:discrete_ap_boundary}
\end{align}
\end{definition}
Here, the discrete linear operator $L_{h,\Delta t}=\Delta_h-\sigma I$
is the linear operator similar to the one introduced in \eqref{fully-discrete-cn}, but is defined now on a larger point set $M^0$.

\begin{remark}
The homogeneous Dirichlet boundary condition~\eqref{eqn:discrete_ap_boundary} in the AP is chosen merely for efficiency of our algorithm, i.e. we employ Fast Poisson Solvers to solve the APs. In general, other boundary conditions can be selected for the AP as long as the defined AP is well-posed and can be solved computationally efficiently.
\end{remark}

\paragraph{Construction of the Particular Solution.} Let us denote by
$G_{h,\Delta t}F_{j,k,l}^{i+1},\;(x_j,y_k,z_l)\in N^+$, the {\it
  Particular Solution} of the fully
discrete problem~\eqref{fully-discrete-cn}.  The Particular Solution is defined on $N^+$ at time level $t^{i+1}$, and is obtained by solving the AP \eqref{eqn:discrete_ap}--\eqref{eqn:discrete_ap_boundary} with the following right hand side:
\begin{align}\label{rhs:particular_solution}
q^{i+1}_{j,k,l}=\left\{
\begin{array}{ll}
F_{j,k,l}^{i+1},&\quad(x_j,y_k,z_l)\in M^+,\\
0,& \quad(x_j,y_k,z_l)\in M^-,
\end{array}
\right.
\end{align}
and by restricting the computed solution from $N^0$ to $N^+$.

\paragraph{Construction of the Difference Potentials and Boundary
  Equations with Projections.} To construct the Difference Potentials, let us first define a linear space $W_{\gamma}$ of all grid functions $w^{i+1}_{\gamma} (x_j,y_k, z_l)$ at $t^{i+1}$ on $\gamma$. The functions are extended by zero to other points in $N^0$ set. These grid functions $w^{i+1}_{\gamma}$ are  called densities on the discrete grid boundary $\gamma$ at the time level $t^{i+1}$.
\begin{definition}\label{def:difference-potentials}
The {\it Difference Potential} associated with a given density $w^{i+1}_{\gamma}\in W_{\gamma}$ is the grid function $P_{N^+\gamma}w_{\gamma}^{i+1}$ defined on $N^+$ at the time level $t^{i+1}$, and is obtained by solving the AP~\eqref{eqn:discrete_ap}--\eqref{eqn:discrete_ap_boundary} with the following right hand side:
\begin{align}\label{rhs:difference_potentials}
q^{i+1}_{j,k,l}&=\left\{
\begin{array}{ll}
0,&\quad(x_j,y_k,z_l)\in M^+,\\
L_{h,\Delta t}[w^{i+1}_{\gamma}],&\quad(x_j,y_k,z_l)\in M^-,
\end{array}
\right.
\end{align}
\end{definition}
and by restricting the solution from $N^0$ to $N^+$.

Next, we will introduce the trace operator.  Given a grid function $w^{i+1}$ defined on the point set $N^+$, we denote by $Tr_{\gamma}w^{i+1}$ the trace or restriction of $w^{i+1}$ from $N^+$ to the discrete grid boundary $\gamma$. Similarly, we define $Tr_{\gamma_{in}}w^{i+1}$ as the trace or restriction of $w^{i+1}$ from $N^+$ to  $\gamma_{in}\subset \gamma$.  We are ready to define an operator $P_{\gamma}:W_{\gamma}\rightarrow W_{\gamma}$ such that $P_{\gamma}w^{i+1}_{\gamma}:=Tr_{\gamma}P_{N^+\gamma}w^{i+1}_{\gamma}$. The operator $P_{\gamma}$ is a projection operator. Now, we will state the key theorem for Difference Potentials Method, which allows us to reformulate the difference equation \eqref{fully-discrete-cn} defined on $M^+$ into equivalent {\it Boundary Equations with Projections} (BEP) defined on the discrete grid boundary $\gamma$ only.
\begin{theorem}[Boundary Equations with Projections (BEP)]\label{thm:full_BEP}
At time $t^{i+1}$, the discrete density $u^{i+1}_{\gamma}$ is the trace of some solution $u^{i+1}$ on $N^+$ to the difference equation~\eqref{fully-discrete-cn}, i.e. $u^{i+1}_{\gamma}:=Tr_{\gamma}u^{i+1}$, if and only if the following BEP holds:
\begin{align}\label{eqn:full_BEP}
    u^{i+1}_{\gamma}-P_{\gamma}u^{i+1}_{\gamma}=G_{h,\Delta t}F^{i+1}_{\gamma}, \quad  (x_j, y_k, z_l)\in \gamma,
\end{align}
where $G_{h,\Delta t}F^{i+1}_{\gamma}:=Tr_{\gamma}G_{h,\Delta t}F^{i+1}_{j, k, l}$ is the trace of the Particular Solution on the discrete grid boundary $\gamma$.
\end{theorem}

\begin{proof}
See  \cite{Ryab} or \cite{MR3954435}.\qed
\end{proof}

\begin{remark}
Note, using that Difference Potential is a linear operator, we can recast (\ref{eqn:full_BEP}) as
\begin{align}\label{eqn:full_BEP1}
    u^{i+1}_{m}-\sum_{\mathfrak{n}\in \gamma}A_{\mathfrak{n}m} u^{i+1}_{\mathfrak{n}}=G_{h,\Delta t}F^{i+1}_m, \quad m\in \gamma,
\end{align}
where $m$ is the index of a grid point in the set $\gamma$, and $G_{h,\Delta t}F^{i+1}_m$ is the value of the Particular Solution at the grid point with index $m$ in the set $\gamma$.
\end{remark}

\begin{proposition}\label{prop:rank}
The rank of linear equations in BEP~\eqref{eqn:full_BEP} is $|\gamma_{in}|$, which is the cardinality of the point set $\gamma_{in}$.
\end{proposition}

\begin{proof}
The proof follows the lines of the proof in \cite{Ryab,MR3954435},
and we will present it below for reader's convenience. If the density
$u^{i+1}_{\gamma_{ex}}$ on $\gamma_{ex}$ to the difference
equation~\eqref{fully-discrete-cn} is given, then such discrete system
will admit a unique solution $u^{i+1}_{j, k, l}$ defined on a set
$N^{+}$. Hence, the BEP~\eqref{eqn:full_BEP} will have a unique
solution, if $u^{i+1}_{\gamma_{ex}}$ is given. Thus,  the solution $u^{i+1}_{\gamma}$ to BEP~\eqref{eqn:full_BEP} has dimension $|\gamma_{ex}|$, which is the cardinality of set $\gamma_{ex}$. As a consequence, the BEP \eqref{eqn:full_BEP} has rank $|\gamma|-|\gamma_{ex}|=|\gamma_{in}|$.\qed
\end{proof}

Next, we introduce the reduced BEP \eqref{eqn:reduced_BEP} defined only
on $\gamma_{in}$ that can be shown to be equivalent to the BEP~\eqref{eqn:full_BEP} defined on $\gamma$.
\begin{theorem}\label{thm:reduced_BEP}
The BEP~\eqref{eqn:full_BEP} defined on $\gamma$ in Theorem~\ref{thm:full_BEP} is equivalent to the following BEP~\eqref{eqn:reduced_BEP} defined on a smaller subset $\gamma_{in}\subset \gamma$:
\begin{align}\label{eqn:reduced_BEP}
    u^{i+1}_{\gamma_{in}}-Tr_{\gamma_{in}}P_{\gamma}u_{\gamma}^{i+1}=Tr_{\gamma_{in}}G_{h,\Delta t}F^{i+1}_{\gamma}, \quad (x_j, y_k, z_l)\in \gamma_{in}
\end{align}
Moreover, the reduced BEP~\eqref{eqn:reduced_BEP} contains only linearly independent equations.
\end{theorem}

\begin{proof}
The proof follows the lines of the proof in \cite{Ryab,MR3954435}
and we will present it below for reader's convenience. First, define the grid function:
\begin{align}\label{eqn:grid-function}
\Phi^{i+1}:=P^{i+1}+G^{i+1}-u^{i+1}_{\gamma},\quad \mbox{on } N^0,
\end{align}
where $P^{i+1}$ is a solution to the AP~\eqref{eqn:discrete_ap}--\eqref{eqn:discrete_ap_boundary} on $N^0$ with right hand side \eqref{rhs:difference_potentials} using density $u^{i+1}_{\gamma}$, $G^{i+1}$ is a solution to the AP~\eqref{eqn:discrete_ap}--\eqref{eqn:discrete_ap_boundary} on $N^0$ with right hand side~\eqref{rhs:particular_solution}, and $u^{i+1}_\gamma$ is extended from $\gamma$ to $N^0$ by zero. By the construction of $\Phi^{i+1}$, one can see that $\Phi^{i+1}$ is a solution to the following difference equation:
\begin{equation}\label{eqn:full-laplace-equation}
\begin{aligned}
L_{h,\Delta t}[\Phi^{i+1}]
&=\left\{
\begin{array}{ll}
F^{i+1}-L_{h,\Delta t}[u^{i+1}_{\gamma}],&\quad\mbox{on } M^+,\\
0, & \quad\mbox{on }M^-.
\end{array}
\right.
\end{aligned}
\end{equation}
Therefore, we conclude that $\Phi^{i+1}$ solves the following homogeneous
difference equations on the set $M^-$:
\begin{align}\label{eqn:laplace-on-Mm}
L_{h,\Delta t}\Phi^{i+1}=0,\quad \mbox{on } M^-.
\end{align}
Additionally, by construction of functions $\Phi^{i+1},P^{i+1}$ and $G^{i+1}$, the grid function $\Phi^{i+1}$ satisfies the following boundary condition:
\begin{align}\label{eqn:bc}
\Phi^{i+1}=0,\quad \mbox{on } N^0\backslash M^0.
\end{align}

Next, observe that the BEP~\eqref{eqn:full_BEP} and the reduced BEP~\eqref{eqn:reduced_BEP} can be reformulated using grid function $\Phi^{i+1}$ in \eqref{eqn:grid-function} as follows:
\begin{align}\label{eqn:equiv-full-BEP}
\Phi^{i+1}=0, \quad\mbox{on }\gamma, \quad(\mbox{BEP }\eqref{eqn:full_BEP}),
\end{align}
and
\begin{align}\label{eqn:equiv-reduced-BEP}
\Phi^{i+1}=0, \quad\mbox{on }\gamma_{in},\quad (\mbox{BEP }\eqref{eqn:reduced_BEP}).
\end{align}
Hence, it is enough to show that \eqref{eqn:equiv-full-BEP} is equivalent to \eqref{eqn:equiv-reduced-BEP} to prove the equivalence between the BEP~\eqref{eqn:full_BEP} and the reduced BEP~\eqref{eqn:reduced_BEP}. First, note that if \eqref{eqn:equiv-full-BEP} is true, then \eqref{eqn:equiv-reduced-BEP} is obviously satisfied. 

Now, assume that \eqref{eqn:equiv-reduced-BEP} is true and let us show that \eqref{eqn:equiv-full-BEP} holds. Consider problem~\eqref{eqn:laplace-on-Mm}: $L_{h,\Delta t}\Phi^{i+1}=0$ on $M^-$, subject to boundary conditions \eqref{eqn:bc} and \eqref{eqn:equiv-reduced-BEP}, since the set $\gamma_{in}\cup(N^0\backslash M^0)$ is the boundary set for set $M^-$. Then we have the following discrete boundary value problem:
\begin{align}\label{eqn:discrete-bvp}
L_{h,\Delta t}\Phi^{i+1}&=0,\quad\mbox{on }M^-,\\
\Phi^{i+1}&=0,\quad \mbox{on }N^0\backslash M^0,\\
\Phi^{i+1}&=0,\quad \mbox{on }\gamma_{in},
\end{align}
which admits a unique zero solution: $\Phi^{i+1}=0$ on $M^-$. Since $\gamma_{ex}\subset M^-$, we conclude that $\Phi^{i+1}=0$ on $\gamma_{ex}$, as well as on $\gamma\equiv\gamma_{ex}\cup\gamma_{in}$, which shows that \eqref{eqn:equiv-reduced-BEP} implies \eqref{eqn:equiv-full-BEP}. 

Thus, we showed that \eqref{eqn:equiv-full-BEP} is equivalent to \eqref{eqn:equiv-reduced-BEP}, and therefore, BEP~\eqref{eqn:full_BEP} is equivalent to the reduced BEP~\eqref{eqn:reduced_BEP}. Moreover, due to Proposition~\ref{prop:rank}, the reduced BEP~\eqref{eqn:reduced_BEP} consists of only linearly independent equations.
\qed
\end{proof}

Similarly to (\ref{eqn:full_BEP})-(\ref{eqn:full_BEP1}), the reduced BEP~\eqref{eqn:reduced_BEP} can be recast as 
\begin{align}\label{eqn:reduced_BEP1}
    u^{i+1}_{m}-\sum_{\mathfrak{n}\in \gamma}A_{\mathfrak{n}m} u^{i+1}_{\mathfrak{n}}=G_{h,\Delta t}F^{i+1}_m, \quad m\in \gamma_{in}.
\end{align}

\begin{remark}
The BEP~\eqref{eqn:full_BEP} or \eqref{eqn:reduced_BEP} reduces degrees of freedom from $\mathcal{O}(h^{-3})$ in the difference equation \eqref{fully-discrete-cn} to $\mathcal{O}(h^{-2})$. In addition, the reduced BEP~\eqref{eqn:reduced_BEP} defined on $\gamma_{in}$ reduces the number of equations in BEP \eqref{eqn:full_BEP} by approximately one half, since $|\gamma_{in}|\approx|\gamma|/2$. Thus, using the reduced BEP~\eqref{eqn:reduced_BEP} will further improve the computational cost in our numerical algorithm, especially in 3D, and we will use the reduced BEP as a part of the proposed numerical algorithm.
\end{remark}
Additionally, let us note that the BEP~\eqref{eqn:full_BEP} or the
BEP~\eqref{eqn:reduced_BEP} will admit multiple solutions since the
system of equations \eqref{eqn:full_BEP} (and hence~\eqref{eqn:reduced_BEP}) is equivalent to the system of
difference equations \eqref{fully-discrete-cn} without imposed
boundary conditions yet. Therefore, to construct a unique solution to
BEP~\eqref{eqn:reduced_BEP}, we need to supply the BEP
\eqref{eqn:reduced_BEP} with either the dynamic boundary
condition~\eqref{eqn:dynamic-bc}, or the coupling 
conditions on the surface~\eqref{eqn:coupling}-\eqref{eqn:surface}. To
impose these conditions efficiently into BEP, we will introduce the
extension operator \eqref{eqn:extension_operator} and combine \eqref{eqn:extension_operator} with the
spectral approach discussed below for the approximation of the boundary
conditions/surface equations.\\
\begin{definition}
The extension operator $\pi_{\gamma \Gamma}[u^{i+1}]$ of the function $u(x,y,z,t^{i+1})$ from a point $(x,y,z)\in\Gamma$ to $(x_j,y_k,z_l)\in\gamma$ is defined as:
\begin{align}\label{eqn:extension_operator}
    \pi_{\gamma \Gamma}[u^{i+1}]|_{(x_j, y_k,z_l)}:=u^{i+1}(x,y,z)|_{\Gamma}+\left.d\frac{\partial u^{i+1}(x,y,z)}{\partial n}\right|_{\Gamma}+\left.\frac{d^2}{2}\frac{\partial^2u^{i+1}(x,y,z)}{\partial n^2}\right|_{\Gamma},
\end{align}
where $n$ is the unit outward normal vector on $\Gamma$, $d$ is the signed distance between a point $(x_j,y_k,z_l)\in \gamma$ and the
point of its orthogonal projection $(x,y,z)$ on the continuous boundary $\Gamma$ in the direction of $n$. 
\end{definition}
Basically, the extension operator~\eqref{eqn:extension_operator}
defines values of $\pi_{\gamma\Gamma}[u^{i+1}]$ at the point of the
discrete grid boundary $(x_j, y_k, z_l)\in \gamma$ with the desired
accuracy through the values of the continuous solution and its
gradients at time $t^{i+1}$ at the continuous boundary $\Gamma$ of the
domain. In particular, we consider the extension operator \eqref{eqn:extension_operator} defined in $(x_j,y_k,z_k)\in\gamma_{in}$ when we solve the reduced BEP~\eqref{eqn:reduced_BEP}. In addition, note that $d$ and $n$ need not to be known precisely, see Tables \ref{table:d-perturbed}--\ref{table:all-perturbed} in Section \ref{sec:numerics-dbc}.\\
\par\noindent {\bf Discretization on the Surface: }
\par Here, for
simplicity, we assume that the surface $\Gamma$ is a sphere with
radius $R$. However, the proposed numerical algorithms can be extended
to more general smooth domains and, hence, more general surfaces, and
the main steps of the methods will stay the same (see Remark \ref{remark:derivatives-shperical-harmonics} below).
\paragraph{Case 1: Dynamic Boundary Conditions
  \eqref{eqn:dynamic-bc}.}
We will use trapezoidal in time scheme for \eqref{eqn:dynamic-bc},
but other time discretizations can be employed as well. 
Since, in this work $\Gamma$ is a sphere, we have that the normal derivative satisfies,
\begin{align}
\frac{\partial u(x, y, z, t)}{\partial n}=\frac{\partial u(x, y, z,t)}{\partial r},\quad (x, y, z)\in\Gamma,
\end{align}
where $n$ is the unit outward normal vector and $r$ is the variable radius in the spherical coordinates, and similarly, $u_{nn}=u_{rr}$.

The discrete in time dynamic boundary condition~\eqref{eqn:dynamic-bc} is
\begin{align}\label{eqn:discrete-dynamic-bc}
&\frac{u^{i+1}(x, y, z)-u^i(x, y, z)}{\Delta t}\nonumber\\
=&\frac{1}{2}\Big(\Delta_\Gamma u^{i+1}(x, y, z)-u^{i+1}(x, y, z)-\frac{\partial u^{i+1}(x, y, z)}{\partial r}+g^{i+1}(x, y, z)\\
&+\Delta_{\Gamma} u^{i}(x, y, z)-u^{i}(x, y, z)-\frac{\partial u^{i}(x, y, z)}{\partial r}+g^{i}(x, y, z) \Big),\nonumber
\end{align}
for $(x, y, z)\in\Gamma$. Here, $u^{i+1}(x, y, z)$ is an approximation in time
of $u(x, y, z, t^{i+1})$, and $g^{i+1}(x, y, z)$ is an approximation of
$g(x, y, z, t^{i+1})$ at time level $t^{i+1}$.
Also, note that, the Laplace-Beltrami operator on the sphere $\Gamma$
with a radius $R$ at time $t^{i+1}$ can be obtained as,
\begin{align}\label{eqn:laplace-beltrami}
\Delta_\Gamma u^{i+1}(x, y,
  z)=\frac{1}{R^2\sin\theta}\frac{\partial}{\partial\theta}\left(\sin\theta\frac{\partial u^{i+1}(x, y,
  z)}{\partial\theta}\right)+\frac{1}{R^2\sin^2\theta}\frac{\partial^2u^{i+1}(x,
  y, z)}{\partial\varphi^2},
\end{align}
where $(\theta,\varphi)$ are the polar and azimuthal angles for a point $(x,y,z)\in\Gamma$. \\
Next, from~\eqref{eqn:discrete-dynamic-bc}, we can express the term $u^{i+1}_r(x, y, z)$ as,
\begin{align}\label{eqn:discrete-dynamic-bc1}
\frac{\partial u^{i+1}(x, y, z)}{\partial r}&=\Delta_\Gamma u^{i+1}(x,
                                              y, z)-(1+\sigma)u^{i+1}(x,y,z)+\sigma
                                              u^{i}(x, y,
                                              z)+g^{i+1}(x, y, z)\nonumber\\
&\quad+\Delta_\Gamma u^{i}(x, y, z)-u^{i}(x, y, z)-\frac{\partial
  u^{i}(x, y, z)}{\partial r}+g^{i}(x, y, z),\quad (x, y, z) \in\Gamma\nonumber\\
&=\Delta_\Gamma u^{i+1}(x, y, z)-(1+\sigma)u^{i+1}(x, y, z)\\
&\quad+\sigma u^{i}(x, y, z)+g^{i+1}(x, y, z)+u_t^{i}(x, y, z),\quad (x,
  y, z)\in\Gamma,\nonumber
\end{align}
where $\sigma=2/\Delta t$ as before, and $u^i_t(x, y, z)$ denotes the time derivative
of $u(x, y, z, t)$ at time level $t^{i}$,
\begin{align}
u_t^{i}(x, y, z)=\Delta_\Gamma u^{i}(x, y,z)-u^{i}(x, y,
  z)-\frac{\partial u^{i}(x, y, z)}{\partial r}+g^{i}(x, y, z).
\end{align}
We assume that $u^{0}_t(x, y, z)$ is known initially, since
$u^0(x, y, z)$ and $g^0(x, y, z)$ are known at the initial time. Note, that the time
derivative $u^{i+1}_t(x, y, z)$ at the next time level $t^{i+1}$, can be updated efficiently using the following
formula (consequence of \eqref{eqn:dynamic-bc} and \eqref{eqn:discrete-dynamic-bc}),
\begin{align}\label{eqn:time-derivative-update}
u^{i+1}_t(x, y, z) = \sigma u^{i+1}(x, y,z)-\sigma u^i(x, y, z)-u^i_t(x, y, z),\quad (x, y, z)\in \Gamma,
\end{align}
once we have computed $u^{i+1}(x, y, z)$ at time level $t^{i+1}$.

Furthermore, we note that $u^{i+1}_{rr}(x, y, z)$ can be expressed in terms of $u^{i+1}_r(x, y, z)$ if one subtracts \eqref{eqn:dynamic-bc} from \eqref{eqn:dynamic-eq} by extending \eqref{eqn:dynamic-eq} outside of domain $\Omega$:
\begin{align}
\frac{\partial^2 u^{i+1}(x,y,z)}{\partial r^2}=-u^{i+1}(x,y,z)-\left(1+\frac{2}{R}\right)\frac{\partial u^{i+1}(x,y,z)}{\partial r}-f^{i+1}(x,y,z)+g^{i+1}(x,y,z),\label{eqn:u_rr-exten-u_r}
\end{align}
for $(x,y,z)\in\Gamma$.
Also, note that the normal derivative $u^{i+1}_r(x, y, z)$ depends linearly on $u^{i+1}(x, y, z)$ as in \eqref{eqn:discrete-dynamic-bc1}, hence we only need to determine one unknown term $u^{i+1}(x, y, z)$ in the extension operator \eqref{eqn:extension_operator}.

\paragraph{Spectral Approach.}
To combine extension operator (\ref{eqn:extension_operator}) accurately
and efficiently with dynamic boundary condition
\eqref{eqn:discrete-dynamic-bc} (and, hence with
\eqref{eqn:discrete-dynamic-bc1}), we will introduce the spectral
approximations at each time level $t^{i+1}$ of the following two terms:
\begin{align}
u^{i+1}(x, y, z)&\approx \sum_{\kappa=1}^{L}a^{i+1}_{\kappa}\phi_{\kappa}(\theta,\varphi),\quad (x, y, z)\in \Gamma,\label{eqn:dbc-spectral-1}
\end{align}
where $(\theta,\varphi)$ are the polar and the azimuthal angles for a
point $(x, y, z)\in\Gamma$.

\begin{remark}
Here, the number of spherical harmonics $L$ does not depend on the underlying mesh sizes and depends on the properties of the solutions to the models.
% In general, number of harmonics for \eqref{eqn:dbc-spectral-1} can be different
% than number of harmonics for \eqref{eqn:dbc-spectral-2} (depends on the
% regularity). Here, for simplicity, we assume the same number of
% harmonics for \eqref{eqn:dbc-spectral-1} and \eqref{eqn:dbc-spectral-2}.
\end{remark}

Now, combining relations \eqref{eqn:discrete-dynamic-bc1}, \eqref{eqn:u_rr-exten-u_r} and
\eqref{eqn:dbc-spectral-1} with the
extension operator~\eqref{eqn:extension_operator}, we obtain
\begin{align}
\pi_{\gamma\Gamma}[u^{i+1}]|_{(x_j,y_k,z_l)} &=u^{i+1}(x, y, z)|_{\Gamma}+d\left.\frac{\partial
                 u^{i+1}(x, y, z)}{\partial
                 r}\right|_{\Gamma}+\frac{d^2}{2}\left.\frac{\partial^2
                 u^{i+1}(x, y, z)}{\partial r^2}\right|_{\Gamma}\\
&=\left(1-d(1+\sigma)+\frac{d^2}{2}\left(\left(\frac{2}{R}+1\right)(1+\sigma)-1\right)\right)u^{i+1}(x, y, z)\nonumber\\
&\quad+\left(d-\frac{d^2}{2}\left(\frac{2}{R}+1\right)\right)\Delta_\Gamma u^{i+1}(x, y, z)\nonumber\\
&\quad+d(\sigma u^i(x, y, z)+g^{i+1}(x, y,
  z)+u_t^{i}(x, y, z))\\
&\quad-\frac{d^2}{2}\left(\left(\frac{2}{R}+1\right)(\sigma u^i(x, y, z)+g^{i+1}(x, y,
  z)+u_t^{i}(x, y, z))\right)\nonumber\\
&\quad+\frac{d^2}{2}\left(-f^{i+1}(x,y,z)+g^{i+1}(x,y,z)\right)\nonumber\\
&\approx u^{i+1}_\gamma (x_j,y_k,z_l)\\
&=A\bm{a}^{i+1}+\bm{c}^{i+1},\quad  (x_j,y_k,z_l)\in
  \gamma \mbox{ and } (x ,y, z)\in \Gamma,\label{eqn:extension-operator-reformulated}
\end{align}
where $\bm{a}^{i+1}$ is the vector of the unknown spectral coefficients $a^{i+1}_{\kappa}$,
% $\bm{b}^{i+1}$ is the vector of the unknown spectral coefficients $b^{i+1}_{\kappa}$ 
$\bm{c}^{i+1}$ denotes the known vector: 
\begin{align}
\bm{c}^{i+1}=&d(\sigma u^i(x, y, z)+g^{i+1}(x, y, z)+u_t^{i}(x, y, z))\nonumber\\
&-\frac{d^2}{2}((\frac{2}{R}+1)(\sigma u^i(x, y, z)+g^{i+1}(x, y, z)+u_t^{i}(x, y, z)))\nonumber\\
&+\frac{d^2}{2}(-f^{i+1}(x,y,z)+g^{i+1}(x,y,z)),
\end{align}
and $d$ is the signed distance from
the point  $(x_j, y_k, z_l)$ in $\gamma$ to its foot point $(x, y, z)$ on the continuous boundary $\Gamma$. The coefficient matrix $A$ is assembled using the basis functions, i.e.,
\begin{align}
A_{m,\kappa} &=
\left(1-d_m(1+\sigma)+\frac{d_m^2}{2}\left(\left(\frac{2}{R}+1\right)(1+\sigma)-1\right)\right)\phi_{\kappa}(\theta_m,\varphi_m)\nonumber\\
&\quad+\left(d_m-\frac{d_m^2}{2}\left(\frac{2}{R}+1\right)\right)\Delta_\Gamma \phi_{\kappa}(\theta_m,\varphi_m)
\end{align}
where $m$ is the index that represents a point in $\gamma$, 
$(\theta_m,\varphi_m)$ are the polar and azimuthal angles for the foot
point $(x, y, z) \in \Gamma$ of a point $m$ in
$\gamma$, and $d_m$ is the signed distance for this point. Note, $A$ is assembled using whole
  point set $\gamma$. However, only the rows corresponding to
  $\gamma_{in}$ will be used in our algorithm when we solve the reduced BEP~\eqref{eqn:reduced_BEP}.

\begin{remark}\label{remark:derivatives-shperical-harmonics}
a) In the special case of a sphere, the surface laplacian of a spherical harmonic is conveniently obtained by the following eigenvalue-eigenfunction relation:
\begin{align}\label{eqn:eigen-approach}
\Delta_{\Gamma}Y_{\ell}^{\mathfrak{m}}(\theta,\varphi) = -\ell(\ell+1)R^2Y_{\ell}^{\mathfrak{m}}(\theta,\varphi)
\end{align}
where $Y_\ell^{\mathfrak{m}}(\theta,\varphi)$ is the spherical harmonic function of degree $\ell$ and order ${\mathfrak{m}}$ (see detailed formulas \eqref{eqn:real-spherical-harmonics}--\eqref{eqn:kappa-relation} in Section \ref {sec:numerics-setup}) and $R$ is the radius of the sphere, 
see also \cite{MR3875514,MR3794068}. Another equivalent approach is to use \eqref{eqn:laplace-beltrami} and \eqref{eqn:dbc-spectral-1}, where the derivatives of the spherical harmonics
$\frac{\partial\phi_{\kappa}^{i+1}}{\partial\theta}$,
$\frac{\partial^2\phi_{\kappa}^{i+1}}{\partial\theta^2}$,$\frac{\partial^2\phi_{\kappa}^{i+1}}{\partial\varphi^2}$
can be obtained using recursive formula \cite{MR1225604}. 
In the numerical section, we adopt the relation \eqref{eqn:eigen-approach} for the efficiency of the codes.

b) In this work, we showcase the versatility of the DPM framework for dynamic BC and bulk-surface problems, and we illustrate the ideas of the method using spherical geometry in 3D. We should note that, the basis functions in the spectral approximation of the terms in the extension operator  in the DPM framework are not limited to spherical harmonics. For example, in the case of smooth geometry other than spheres, local radial basis functions can be employed instead of spherical harmonics. In addition, DPM-based algorithms were developed for models on domains with corners (2D) \cite{MR3651387} or wedges (3D) \cite{MR3954435}. Furthermore, one possible future direction is to replace the spectral approximation on the surface with a more general method that can handle arbitrary geometry, for instance using ideas of the trace finite element method (trace-FEM) \cite{MR3345245,MR3717149} that utilizes the restriction (trace) of a volumetric finite element space of piecewise continuous trilinear functions, to solve surface equations. In addition, the choice of the discretization of the bulk equation in the DPM framework has also a flexibility (and can be selected to be FEM, for example).

\end{remark}

\paragraph{Case 2: Bulk-Surface Coupling \eqref{eqn:coupling}-\eqref{eqn:surface}.} As for the bulk-surface problems, we assume here
that the surface $\Gamma$ is also a sphere with radius $R$, and thus, the
Laplace-Beltrami operator $\Delta_\Gamma$ at time $t^{i+1}$ is computed using the eigenvalue approach \eqref{eqn:eigen-approach}. Again, the first order normal
derivative is computed as $u_n(x, y, z, t)\equiv \nabla u(x, y, z,
t)\cdot n = u_r(x, y, z, t)$ for $(x, y, z)\in\Gamma$. 
\par To discretize in time equation on the surface
\eqref{eqn:surface}, we will use trapezoidal in time scheme
as it is used in the bulk \eqref{fully-discrete-cn}. The discrete in
time surface equation is (as a result of \eqref{eqn:surface}):
\begin{align}\label{eqn:semi-discrete-surface-eq}
\frac{v^{i+1}(x, y, z)-v^i(x, y, z)}{\Delta t}
  &=\frac{1}{2}(\Delta_\Gamma v^{i+1}(x, y, z)+g^{i+1}(x, y,
  z)+h(u^{i+1}(x, y, z),v^{i+1}(x, y, z))\nonumber\\
&+v^{i}_t(x, y, z)), \quad (x,
  y, z)\in\Gamma, 
\end{align}
where $v^{i}_t=\Delta_\Gamma v^{i}+g^{i}+h(u^{i},v^{i})$. Note that, to compute the term $v^i_t$ efficiently, we use the formula,
\begin{align}\label{eqn:vti}
v^{i}_t(x, y, z)=\sigma v^{i}(x, y,z)-\sigma v^{i-1}(x, y, z)-v^{i-1}_t(x, y, z), \quad (x, y, z)\in \Gamma,
\end{align} 
which is consequence of the discretization
\eqref{eqn:semi-discrete-surface-eq} and \eqref{eqn:surface}.
Moreover, since from \eqref{eqn:coupling}, we have that $h(u^{i+1},
v^{i+1})=-u^{i+1}_r$, we obtain the following expression for
$u^{i+1}_r$,
\begin{align}
\frac{\partial u^{i+1}(x, y, z)}{\partial r} &= -\sigma v^{i+1}(x, y,
                                               z)+\Delta_\Gamma
                                               v^{i+1}(x, y,
                                               z)\nonumber\\
&+\sigma v^i(x, y, z)+g^{i+1}(x,
                                               y, z)+v^i_t(x, y, z),
                                               \quad (x, y, z)\in\Gamma, \label{eqn:normal-derivative-in-terms-of-v}
\end{align}
where as before, $\sigma=2/\Delta t$.
\par a) {\it Linear Bulk-Surface Coupling.} For simplicity, we first consider case
of linear coupling function $h(u, v)$ in  \eqref{eqn:coupling} similar
to, for example, \cite{Burman_2015} and \cite{Elliott_2012},
\begin{align}\label{eqn:lin-coup}
h(u, v) = u-v, \mbox{ on } \Gamma.
\end{align}
Since $h(u^{i+1},v^{i+1})=u^{i+1}-v^{i+1}$ at time level $t^{i+1}$, and using equation
\eqref{eqn:semi-discrete-surface-eq}, we have that,
\begin{align}
u^{i+1}(x, y, z) &= (1+\sigma)v^{i+1}(x, y, z)-\Delta_\Gamma
                   v^{i+1}(x, y, z)\nonumber\\
&-\sigma v^i(x, y, z) - g^{i+1}(x,
                   y, z)-v^i_t(x, y, z), \quad (x, y, z)\in\Gamma.\label{eqn:in-terms-of-v}
%\end{align}
%\begin{align}
\end{align}
\paragraph{Spectral Approach.} Similarly to model with dynamic boundary conditions, to couple accurately and efficiently discretization of the bulk
equations, hence, the reduced BEP \eqref{eqn:reduced_BEP} with the
discretization of the surface equation \eqref{eqn:surface} combined
with coupling function \eqref{eqn:lin-coup}, we will
employ idea of extension operator \eqref{eqn:extension_operator}
together with the spectral approximation of the functions $v^{i+1}(x,
  y, z)$ and $\frac{\partial^2 u^{i+1}(x, y, z)}{\partial r^2}$, $(x,
  y, z)\in \Gamma$ at each time level $t^{i+1}$.
\par Hence, for the density $u^{i+1}_\gamma$, we combine the extension
operator~\eqref{eqn:extension_operator} together with relations
\eqref{eqn:normal-derivative-in-terms-of-v}-\eqref{eqn:in-terms-of-v},
to obtain:
\begin{align}
\pi_{\gamma\Gamma}[u^{i+1}]|_{(x_j,y_k,z_l)} =& u^{i+1}(x, y, z) + d\frac{\partial
                                 u^{i+1}(x, y, z)}{\partial
                                 r}+\frac{d^2}{2}\frac{\partial^2
                                 u^{i+1}(x, y, z)}{\partial r^2}\\
=&[(1+\sigma)v^{i+1}-\Delta_\Gamma v^{i+1}]+d\left[-\sigma v^{i+1}+\Delta_\Gamma v^{i+1}\right]+\frac{d^2}{2}\frac{\partial^2 u^{i+1}}{\partial r^2}\nonumber\\
&+[-\sigma v^i- g^{i+1}-v^i_t]+d\left[\sigma
  v^i+g^{i+1}+v^i_t\right],\label{eqn:extension-operator-bs}
\end{align}
where $(x, y, z) \in \Gamma$ is the foot point of a point $(x_j, y_k, z_l)$ in
the discrete grid boundary $\gamma$, and $d$ is the signed distance
from a point $(x_j, y_k, z_l)$ in $\gamma$ to its foot point $(x, y,
z) \in\Gamma$.

\par Next, similarly to the approximation of the dynamic boundary
conditions, to construct density $u^{i+1}_\gamma $ efficiently for the bulk model \eqref{eqn:bulk}, we
assume spectral approximations of the terms $v^{i+1}(x, y, z)$ and, also of the term
$\frac{\partial^2 u^{i+1}(x, y, z)}{\partial r^2}$ in the extension operator~\eqref{eqn:extension-operator-bs}, i.e.,
\begin{align}
v^{i+1}(x, y, z)&\approx\sum_{\kappa=1}^{L}a^{i+1}_{\kappa}\phi_{\kappa}(\theta,\varphi), \quad (x, y, z)\in \Gamma,\\
\frac{\partial^2 u^{i+1}(x, y, z)}{\partial r^2}&\approx\sum_{\kappa=1}^{L}b^{i+1}_{\kappa}\phi_{\kappa}(\theta,\varphi), \quad (x, y, z)\in \Gamma,
\end{align}
where $\theta$ and $\varphi$ are the polar and the azimuthal angles of
the point $(x, y, z) \in \Gamma$. Then, after we replace
$v^{i+1}$ and $\frac{\partial^2 u^{i+1}(x, y, z)}{\partial r^2}$ in \eqref{eqn:extension-operator-bs} using the spectral approximations above, the approximation to the extension operator \eqref{eqn:extension-operator-bs} is given by,
\begin{align}\label{eqn:extension-operator-bs-reformulated}
\pi_{\gamma\Gamma}[u^{i+1}]|_{(x_j,y_k,z_l)} \approx u^{i+1}_\gamma = A\bm{a}^{i+1}+B\bm{b}^{i+1}+\bm{c}^{i+1},
\end{align}
where $\bm{a}^{i+1}$, $\bm{b}^{i+1}$ are the vectors that store the
unknown spectral coefficients, and $\bm{c}^{i+1}$ denotes the known term:
\begin{align}
\bm{c}^{i+1}&=[-\sigma v^i(x, y, z) - g^{i+1}(x, y, z)-v^i_t(x, y,
  z)]\nonumber\\
&+d\left[\sigma v^i(x, y, z)+g^{i+1}(x, y, z)+v^i_t(x, y,
  z)\right], \quad (x, y, z) \in \Gamma.
\end{align}
The coefficient matrices $A$ and $B$ are computed as,
\begin{align}
A_{m,\kappa} &= (1+\sigma)\phi_{\kappa}(\theta_m,\varphi_m)-\Delta_\Gamma \phi_{\kappa}(\theta_m,\varphi_m)\nonumber\\
&\quad+ d_m\left[-\sigma \phi_{\kappa}(\theta_m,\varphi_m)+\Delta_\Gamma \phi_{\kappa}(\theta_m,\varphi_m)\right],\\
B_{m,\kappa} &= \frac{d^2_{m}}{2}\phi_{\kappa}(\theta_m,\varphi_m).
\end{align}
Here $m$ is the index that represents a point in $\gamma$,
$(\theta_m,\varphi_m)$ are the polar and azimuthal angles for the foot
point $(x, y, z)\in \Gamma$ of a point $m$ in
$\gamma$, and $d_m$ is the signed distance for this point. Similarly, matrices $A$ and $B$ are assembled
  using the whole point set $\gamma$, but only the rows corresponding
  to the $\gamma_{in}$ set will be used in our algorithm to solve the reduced BEP~\eqref{eqn:reduced_BEP}.
\par b) {\it Nonlinear Bulk-Surface Coupling.} Here, we consider the example of
nonlinear coupling function $h(u, v)$ in \eqref{eqn:coupling}, similar
to, for example, \cite{Elliott_2017},
\begin{align}\label{eqn:nonlinear-coupling}
h(u,v) = uv.
\end{align}
And, as before, at time level $t^{i+1}$, we will have
$u^{i+1}v^{i+1}=-u^{i+1}_r$.
\paragraph{Spectral Approach.}
Similar to model with linear bulk-surface coupling \eqref{eqn:lin-coup}, to couple accurately and efficiently discretization of the bulk
equations, hence, the reduced BEP \eqref{eqn:reduced_BEP} with the
discretization of the surface equation \eqref{eqn:surface}, we will
employ idea of extension operator \eqref{eqn:extension_operator}
together with the spectral approximation of the functions $v^{i+1}(x,
  y, z)$, $u^{i+1}(x,
  y, z)$ and $\frac{\partial^2 u^{i+1}(x, y, z)}{\partial r^2}$, $(x,
  y, z)\in \Gamma$ at the time level $t^{i+1}$, i.e.,
\begin{align}
v^{i+1}(x, y,
  z)&\approx\sum_{\kappa=1}^{L}a^{i+1}_{\kappa}\phi_{\kappa}(\theta,\varphi),\quad
        (x, y, z) \in \Gamma,
      \label{sp1}\\
u^{i+1}(x, y,
  z)&\approx\sum_{\kappa=1}^{L}c^{i+1}_{\kappa}\phi_{\kappa}(\theta,\varphi),\quad
        (x, y, z) \in \Gamma,
  \label{sp2}\\
 \frac{\partial^2 u^{i+1}(x, y, z)}{\partial
  r^2}&\approx\sum_{\kappa=1}^{L}b^{i+1}_{\kappa}\phi_{\kappa}(\theta,\varphi), \label{sp3}\quad
        (x, y, z) \in \Gamma,
\end{align}
where, as before, ($\theta$, $\varphi$) are the polar and the azimuthal angles of
the point $(x, y, z) \in \Gamma$. \\
Then, the extension operator \eqref{eqn:extension_operator} becomes,
\begin{align}
\pi_{\gamma\Gamma}[u^{i+1}]|_{(x_j,y_k,z_l)} 
 =& u^{i+1}(x, y, z)+
                                  d\frac{\partial u^{i+1}(x, y,
                                  z)}{\partial
                                  r}+\frac{d^2}{2} \frac{\partial^2
                                  u^{i+1}(x, y, z)}{\partial r^2}\\
=&u^{i+1}+d\left(-\sigma v^{i+1}+\Delta_\Gamma v^{i+1}+\sigma v^i+g^{i+1}+v^i_t\right)+\frac{d^2}{2}\frac{\partial^2 u^{i+1}}{\partial r^2}\\
=&u^{i+1}+d\left(-\sigma v^{i+1}+\Delta_\Gamma v^{i+1}\right)+\frac{d^2}{2}\frac{\partial^2 u^{i+1}}{\partial r^2} +d(\sigma v^i+g^{i+1}+v^i_t)\\
\approx&u^{i+1}_\gamma (x_j, y_k, z_l) \nonumber\\
=&A\bm{a}^{i+1}+B\bm{b}^{i+1}+C\bm{c}^{i+1}+\bm{d}^{i+1},  (x_j, y_k, z_l)\in \gamma  \mbox{ and
  }(x, y, z)\in\Gamma,\label{eqn:extension-op2}
\end{align}
where the coefficient matrices $A,B,C$ for the unknown spectral coefficients $\bm{a}^{i+1},\bm{b}^{i+1},\bm{c}^{i+1}$ are computed as,
\begin{align}
A_{m,\kappa} &= d_m\left[-\sigma \phi_{\kappa}(\theta_m,\varphi_m)+\Delta_\Gamma \phi_{\kappa}(\theta_m,\varphi_m)\right],\\
B_{m,\kappa} &= \frac{d^2_{m}}{2}\phi_{\kappa}(\theta_m,\varphi_m),\\
C_{m,\kappa} &= \phi_{\kappa}(\theta_m,\varphi_m).
\end{align}
Here, $m$ is the index that represents a point in $\gamma$,
$(\theta_m,\varphi_m)$ are the polar and azimuthal angles for the foot
point $(x, y, z)\in \Gamma$ of a point $m$ in $\gamma$, and $d_m$ is the signed distance for this point. The vector
$\bm{d}^{i+1}$ in \eqref{eqn:extension-op2} represents the known quantity,
\begin{align}
\bm{d}^{i+1}=d(\sigma v^i(x,y,z)+g^{i+1}(x,y,z)+v^i_t(x,y,z)),\quad (x,y,z)\in\Gamma,
\end{align}
and is computed at the same foot point $(x, y, z)\in \Gamma$ of a point $m$ in $\gamma$. Again, matrices $A$, $B$ and $C$ are assembled for the entire point set $\gamma$, but only the rows corresponding to the $\gamma_{in}$ set will be used to solve the reduced BEP~\eqref{eqn:reduced_BEP}.
\paragraph{Linearization of the nonlinear coupling~\eqref{eqn:nonlinear-coupling}.}
To efficiently combine the coupling
equation~\eqref{eqn:nonlinear-coupling} with the BEP
\eqref{eqn:reduced_BEP} and with the discretization of the surface
equation \eqref{eqn:semi-discrete-surface-eq}, we will consider
linearization of \eqref{eqn:nonlinear-coupling} at time level $t^{i+1}$.
 
To linearize, we replace $v^{i+1}(x, y, z)$ in \eqref{eqn:nonlinear-coupling} at the time level $t^{i+1}$ by the following
approximation in time
\begin{align}\label{eqn:2-term-approx}
v^{i+1}(x, y, z)= v^i(x, y, z)+\Delta t v^i_t(x, y, z)+\mathcal{O}(\Delta t^2)
\end{align}
where $\Delta t=\mathcal{O}(h)$. Then, the linearization of \eqref{eqn:nonlinear-coupling} gives us,
\begin{align}
-n\cdot\nabla u^{i+1}(x, y, z) \approx u^{i+1}(x, y, z)(v^i(x, y, z)+\Delta
  t v^i_t(x, y, z)),\quad (x, y, z)\in\Gamma,
\end{align}
where $v_t^i$ term is computed via the relation \eqref{eqn:vti}. 
Note, that using \eqref{eqn:normal-derivative-in-terms-of-v} together
with spectral approximation in \eqref{sp1}-\eqref{sp2}, we can
formulate coupling relation \eqref{eqn:nonlinear-coupling} at
$t^{i+1}$ as,
\begin{align}
\Rightarrow -(-\sigma v^{i+1}(x, y, z)+\Delta_\Gamma v^{i+1}(x, y, z)+\sigma v^i(x, y, z)+g^{i+1}(x, y, z)+v_t^i(x, y, z))\nonumber\\
=u^{i+1}(x, y, z)(v^{i}(x, y, z)+\Delta tv^i_t(x, y, z)),\\
\Rightarrow A'a^{i+1}-\sigma v^i(x, y, z)-g^{i+1}(x, y, z)-v^i_t(x, y, z) = C'c^{i+1},\\
\Rightarrow -A'a^{i+1}+C'c^{i+1} = -\sigma
  v^i(x, y, z)-g^{i+1}(x, y, z)-v^i_t(x, y, z). \label{eqn:nonlinear-coupling-spect}
\end{align}
The expression (\ref{eqn:nonlinear-coupling-spect}) gives the linear relation between unknown spectral coefficients $a^{i+1}_\kappa$ and $c^{i+1}_\kappa$. Here, the matrices $A'$ and $C'$ are defined as,
\begin{align}
A'_{m,\kappa} &= -(-\sigma \phi_{\kappa}(\theta_m,\varphi_m)+\Delta_\Gamma \phi_{\kappa}(\theta_m,\varphi_m)),\\
C'_{m,\kappa} &= \phi_{\kappa}(\theta_m,\varphi_m)(v^{i}+\Delta tv^i_t).
\end{align}
Here,  $(\theta_m, \varphi_m)$ corresponds to the angles of the foot
point $(x, y, z)\in \Gamma$ of a point $m$ in $\gamma_{in}$ (since we employ the reduced BEP), and $v^i+\Delta t v^i_t$ is the
corresponding value for the same foot point $m$.

\begin{remark}
One possible improvement is to approximate $v^{i+1}$ in \eqref{eqn:nonlinear-coupling} at $t^{i+1}$ using the
following higher order in time approximation:
\begin{align}\label{eqn:3-term-approx}
v^{i+1}(x, y, z)\approx v^{i}(x, y, z)+\Delta t v^i_t(x, y,
  z)+\frac{\Delta t^2}{2}v^i_{tt}(x, y, z), \quad (x, y, z)\in \Gamma, 
\end{align}
where $v^{i}_{tt}(x, y, z)$ can be approximated using the finite difference approximation in time.
\end{remark}
\noindent\par{\bf Reconstruction of the Solutions at time $t^{i+1}$:}
\paragraph{Case 1: Dynamic Boundary Conditions.}
Next, we use the reduced BEP
\eqref{eqn:reduced_BEP} combined with the approximation of the extension operator in the form \eqref{eqn:extension-operator-reformulated}, to obtain the least squares (LS) system of dimension $|\gamma_{in}|\times L$ for the unknown spectral coefficients $\bm{a}^{i+1}$,
\begin{align}
[A-P_\gamma A]\bm{a}^{i+1} =
  G_{h,\Delta t}F^{i+1}_\gamma-(\bm{c}^{i+1}-P_\gamma\bm{c}^{i+1}), \mbox{ on }
  \gamma_{in}. \label{eqn:BEP-reformulated}
\end{align}
After that, we solve for the unknown spectral coefficients $\bm{a}^{i+1}$, using the normal equation of the reformulated
BEP~\eqref{eqn:BEP-reformulated}. 

\paragraph{Case 2: a) Linear Bulk-Surface Coupling.}
Similarly to the model with dynamic boundary conditions, we combine the reduced BEP
\eqref{eqn:reduced_BEP} and the approximation of the extension operator in the form
\eqref{eqn:extension-operator-bs-reformulated}, to obtain the LS system of dimension $|\gamma_{in}|\times(2L)$ for
the unknown spectral coefficients $\bm{a}^{i+1}$
and $\bm{b}^{i+1}$,
\begin{align}
[A-P_{\gamma}A]\bm{a}^{i+1}+[B-P_{\gamma}B]\bm{b}^{i+1}
  = G_{h,\Delta t}F^{i+1}_\gamma-(\bm{c}^{i+1}-P_{\gamma}\bm{c}^{i+1}), \mbox{
  on } \gamma_{in}.\label{eqn:linear-coupling-reformulated-bep}
\end{align}
Again, we solve for the unknown spectral coefficients $\bm{a}^{i+1}$
and $\bm{b}^{i+1}$ using the normal equation of the reformulated
BEP~\eqref{eqn:linear-coupling-reformulated-bep}.

\paragraph{Case 2: b) Nonlinear Bulk-Surface Coupling.}
Similarly to the model with dynamic boundary conditions and bulk-surface
model with linear coupling, we combine the reduced BEP
\eqref{eqn:reduced_BEP}, the approximation to the extension operator in the form
\eqref{eqn:extension-op2} and the coupling
condition \eqref{eqn:nonlinear-coupling-spect}, to obtain the LS system of dimension $2|\gamma_{in}|\times(3L)$ for the unknown spectral coefficients $\bm{a}^{i+1}$, $\bm{b}^{i+1}$ and $\bm{c}^{i+1}$,
\begin{align}
[A-P_\gamma A]\bm{a}^{i+1}+[B-P_\gamma B]\bm{b}^{i+1}+[C-P_\gamma C]\bm{c}^{i+1}&=G_{h,\Delta t}F^{i+1}_\gamma-(\bm{d}^{i+1}-P_\gamma \bm{d}^{i+1}),\mbox{ on } \gamma_{in},\label{eqn:LS-system}\\
-A'\bm{a}^{i+1}+C'\bm{c}^{i+1} &= -\sigma v^i-g^{i+1}-v^i_t, \mbox{ on } \gamma_{in}\label{eqn:LS-system-2}.
\end{align}
Similarly, we solve for the unknown spectral coefficients $\bm{a}^{i+1}$, $\bm{b}^{i+1}$ and $\bm{c}^{i+1}$, using the normal equation of the LS system \eqref{eqn:LS-system}--\eqref{eqn:LS-system-2}.

\begin{remark}
For the LS system in \textit{Case 1}, \textit{Case 2: a)} and \textit{Case 2: b)} described above, the normal equation approach reduces the computational cost of the algorithms significantly, since the size of the normal matrices will be $L\times L$, $2L\times(2L)$ or $3L\times(3L)$, and $|\gamma_{in}|\gg L$. 
As for the condition numbers of the normal matrices, they can be reduced to the magnitude of approximately $10^3$ on all meshes when one, for example, uses a simple preconditioner based on the maximum value in the column scaling in the LS system, i.e. for LS system $Ax=b$, the normal matrix is $P^TA^TAP$ where $P$ is a diagonal matrix with $P_{ii}=1/max(A_i)$, where $A_i$ is the $i$-th column of the matrix $A$. See Tables \ref{table:condition-number} for examples of the condition numbers.
\end{remark}

\par Once we get the spectral coefficients (see \textit{Case 1}, \textit{Case 2: a)} and \textit{Case 2: b)}), we will be able to reconstruct (i) the solutions $u^{i+1}(x,y,z)$ or $v^{i+1}(x,y,z)$ for $(x,y,z)$ on the surface at the time level $t^{i+1}$ using the spectral approximations; and (ii)
the density $u^{i+1}_\gamma$ at time level $t^{i+1}$ using \eqref{eqn:extension-operator-reformulated}
(dynamic boundary conditions),
\eqref{eqn:extension-operator-bs-reformulated} (bulk-surface model
with linear coupling), or \eqref{eqn:extension-op2} (bulk-surface model
with nonlinear coupling).
Finally, the approximated solution $u^{i+1}_{j,k,l}$, $(x_j,y_k,z_l)\in N^+$  to the model
\eqref{eqn:dynamic-eq}-\eqref{eqn:dynamic-ic} or \eqref{eqn:bulk}-\eqref{eqn:surface-ic} at the time level $t^{i+1}$ is obtained using the
discrete generalized Green's formula
\eqref{eqn:generalized_greens_formula} below.

\paragraph{Discrete Generalized Green's Formula.}  The final step of
DPM is to use the computed density $u^{i+1}_{\gamma}$ to construct the
approximation to the continuous solution in the bulk of the model
\eqref{eqn:dynamic-eq}-\eqref{eqn:dynamic-ic}, or of \eqref{eqn:bulk}-\eqref{eqn:surface-ic}.
\begin{proposition}[Discrete Generalized Green's formula.]\label{prop:discrete_gene_green_formula}
The discrete solution $u^{i+1}_{j, k, l}$ on $N^{+}$ constructed using {\it Discrete Generalized Green's formula},
\begin{align}\label{eqn:generalized_greens_formula}
    u^{i+1}_{j,k,l}=P_{N^+\gamma}u^{i+1}_{\gamma}+G_{h,\Delta t}F^{i+1}_{j,k,l},\quad(x_j,y_k,z_l)\in N^+,
\end{align}
is the approximation to the exact solution $u$ at $(x_j, y_k, z_l) \in
\Omega$ at time  $t^{i+1}$ of the continuous model
\eqref{eqn:dynamic-eq}-\eqref{eqn:dynamic-ic}, or of \eqref{eqn:bulk}-\eqref{eqn:surface-ic}.  We also conjecture that we have the following accuracy of the proposed numerical scheme,
\begin{equation}\label{eqn:accuracy}
    \left|\left|u^{i+1}_{j,k,l}-u(x_j,y_k,z_l,t^{i+1})\right|\right|_{\infty}=\mathcal{O}(h^2+\Delta t^2).
\end{equation}
\end{proposition}
\begin{remark}
The accuracy (\ref{eqn:accuracy}) is observed in all numerical
experiments presented in Section~\ref{sec:numerics}. The reader can consult
\cite{Ryab} for the detailed theoretical foundation of DPM.
\end{remark}

%%%%%%%%%%%%%%%%%%%%%%%%%%%%%%%%%%%%%%%%%%%%%%%%%%%%%%%%%%%%%%%%%%%%%%%%%%%%%%%%%%%%%%
%\paragraph{An Outline of Main Steps of the DPM-based Algorithm.}
\begin{algorithm}
\caption{An Outline of Main Steps of the DPM-based Algorithm}\label{alg:DPM-based-algorithm}
\begin{algorithmic}[1]
\STATE Construct point sets $M^\pm,M^0,N^\pm,N^0$, $\gamma_{ex}$ and $\gamma_{in}$ from uniform meshes on the auxiliary domain $\Omega^0$, which embeds $\Omega$
\STATE Assemble matrices for the reduced BEP:
    \begin{ALC@g}
    \IF {\textit{Case 1}}
        \STATE Assemble $A$, then compute $A-P_\gamma A$ with restriction to the point set $\gamma_{in}$ in \eqref{eqn:BEP-reformulated}
    \ELSIF {\textit{Case 2: a)}}
        \STATE Assemble $A$ and $B$, then compute $A-P_\gamma A,B-P_\gamma B$ with restriction to the point set $\gamma_{in}$ in \eqref{eqn:linear-coupling-reformulated-bep}
    \ELSIF {\textit{Case 2: b)}}
        \STATE Assemble $A$, $B$ and $C$, then compute $A-P_\gamma A,B-P_\gamma B,C-P_\gamma C$ with restriction to the point set $\gamma_{in}$ in \eqref{eqn:LS-system}, and assemble $A'$ in $\gamma_{in}$
    \ENDIF
    \end{ALC@g}
\IF{\textit{Case 1} \OR \textit{Case 2: a)}}
    \STATE Precompute the inverse of the coefficient matrix in the normal equation of the LS system \eqref{eqn:BEP-reformulated} or \eqref{eqn:linear-coupling-reformulated-bep}, using Cholesky decomposition
\ENDIF
\STATE Initialize the bulk/surface solutions using the initial conditions
\WHILE {$t^{i+1}\leq T_{final}$}
\IF {\textit{Case 2: b)}}
\STATE Assemble matrix $C'$ in $\gamma_{in}$ and compute the Cholesky decomposition of the coefficient matrix of the normal equation corresponding to the LS system \eqref{eqn:LS-system}--\eqref{eqn:LS-system-2}
\ENDIF
\STATE Construct the Particular Solution $G_{h,\Delta t}F^{i+1}_{j,k,l}$ on $N^+$ using the discrete AP
\STATE Solve the BEP for the unknown spectral coefficients using the normal equations
\STATE Reconstruct the density $u^{i+1}_\gamma$ using extension operator~\eqref{eqn:extension-operator-reformulated} for \textit{Case 1}, \eqref{eqn:extension-operator-bs-reformulated} for \textit{Case 2: a)}, or \eqref{eqn:extension-op2} for \textit{Case 2: b)}
\STATE Obtain bulk solution $u^{i+1}$ using the discrete generalized Green's formula~\eqref{eqn:generalized_greens_formula}, and surface solution $u^{i+1}$ or $v^{i+1}$ using the spectral approximation
\STATE Update and march in time
\ENDWHILE
\end{algorithmic}
\end{algorithm}

\begin{remark}
We solve the LS systems \eqref{eqn:BEP-reformulated} in \textit{Case 1}, \eqref{eqn:linear-coupling-reformulated-bep} in \textit{Case 2: a)}, and \eqref{eqn:LS-system}--\eqref{eqn:LS-system-2} in \textit{Case 2: b)} using the normal equation approach.
For the normal equations of the resulting algebraic systems, the inverse matrices of the normal matrices are pre-computed outside of the time loop for \textit{Case 1} and \textit{Case 2: a)} using Cholesky decomposition.

For \textit{Case 2: b)}, the normal matrix needs to be assembled and the Cholesky decomposition is performed at each time step since the matrix $C'$ is updated at each time level inside the time loop. However, if the size of the normal matrix is large, for efficiency,  one can exploit the block structures of the normal matrix and update only the blocks associated with $C'$ at each time step.
\end{remark}

%%%%%%%%%%%%%%%%%%%%%%%%%%%%%%%%%%%%%%%%%%%%%%%%%%%%%%%%%%%%%%%%%%%%%%%%%%%%%%%%%%%%%%%%%%%%%%%%%%%%%%%
%%%%%%%%%%%%%% Numerical Section
%%%%%%%%%%%%%%%%%%%%%%%%%%%%%%%%%%%%%%%%%%%%%%%%%%%%%%%%%%%%%%%%%%%%%%%%%%%%%%%%%%%%%%%%%%%%%%%%%%%%%%%
\section{Numerical Results}\label{sec:numerics}
In this section, we illustrate {setup of the numerical tests and present the numerical results (errors and convergence rates, 3D views of the bulk/surface solutions, etc.) for the models with dynamic boundary condition (BC) \eqref{eqn:dynamic-eq}--\eqref{eqn:dynamic-ic}, and for the bulk-surface problems \eqref{eqn:bulk}--\eqref{eqn:surface-ic}. In this work, we restrict our discussion to a spherical domain with radius $R$ centered at the origin. For a general domain in 3D, the proposed algorithms can be extended in a straightforward way, for example, by selecting a different set of basis functions or replacing the spectral approach on the surface with the trace-FEM \cite{MR3345245,MR3717149} (see Remark \ref{remark:derivatives-shperical-harmonics} in Section \ref{sec:algorithm}), which will be reported in future work.

\subsection{Setup of Numerical Tests}\label{sec:numerics-setup}

The auxiliary domain is chosen to be a cube, i.e., $[-R-R/5,R+R/5]\times [-R-R/5,R+R/5] \times [-R-R/5,R+R/5]$. Then, the auxiliary domain is discretized using meshes of dimension $N\times N\times N$ and the grid spacing of the mesh is $h=2(R+R/5)/N$. We adopt the notation $N\times N\times N$ for meshes throughout this numerical section. Note that, other choices of the auxiliary domains will also work.

For the basis functions $\phi_\kappa(\theta,\varphi)$, we use the following spherical harmonics:
\begin{align}\label{eqn:real-spherical-harmonics}
Y_\ell^\mathfrak{m}(\theta,\varphi) = \left\{
\begin{array}{lc}
P_\ell^\mathfrak{m}(\cos\theta), &\quad \mathfrak{m}=0,\\
P_\ell^\mathfrak{m}(\cos\theta)\cos(m\varphi), &\quad \mathfrak{m}>0,\\
P_\ell^{|\mathfrak{m}|}(\cos\theta)\sin(|m|\varphi), &\quad \mathfrak{m}<0,\\
\end{array}\right.
\quad\mbox{for }-\ell\leq \mathfrak{m} \leq \ell.
\end{align}
where $Y_\ell^{\mathfrak{m}}(\theta,\varphi)$ is the spherical harmonic function of degree $\ell$ and order ${\mathfrak{m}}$. 
For the index $\kappa$ in $\phi_\kappa(\theta,\varphi)$, it is related to $(\ell,\mathfrak{m})$, i.e.,
\begin{align}\label{eqn:kappa-relation}
\kappa=\left\{
\begin{array}{lc}
\ell^2+2\mathfrak{m}+1,& \mathfrak{m}\geq0,\\
\ell^2+2|\mathfrak{m}|,& \mathfrak{m}<0,\\
\end{array}
\right.
\end{align}

The total number of spherical harmonics used in the tests is determined by the exact solutions $u(x,y,z,t)$ and $v(x,y,z,t)$ on the boundary $\Gamma$. Generally, the spectral coefficients of the spherical harmonic basis functions for the initial data of $u$ and $v$ can be computed. This helps to determine the degree and the order of the spherical harmonics to be included in the spectral approximations. Thus, the total number of harmonics used in the numerical tests is independent of the grid spacing $h$. The only constraint on the number of the harmonics is that, the total number of unknown spectral coefficients in the BEPs (\eqref{eqn:BEP-reformulated} for \textit{Case 1}, \eqref{eqn:linear-coupling-reformulated-bep} for \textit{Case 2: a)}, and \eqref{eqn:LS-system} for \textit{Case 2: b)}) is much less than $|\gamma_{in}|$. Generally, this condition is easily satisfied due to the abundance of mesh nodes in $\gamma_{in}$ in 3D, and the relative small number of basis functions required to resolve $u$ and $v$ on the boundary.

In all the numerical tests in this section, we set the final time to
be $T=0.1$. For the time approximation of the models, we adopt the
second-order trapezoidal scheme, and we use the time step $\Delta t = h$, since we consider the second-order approximation in space. There is no particular reason of the choice of the trapezoidal rule, and other second-order implicit time stepping techniques can also be employed. For example, one can use the second-order implicit Runge-Kutta scheme, and the numerical results will not be significantly different from the ones obtained with the trapezoidal rule.

\subsection{The Bulk/Surface Errors}

The approximation to the $\infty$-, $L^2$- and $H^1$-norm errors in the bulk are computed using the following formulas respectively:
\begin{align}
||u-u^i_h||_{\infty(\Omega)}\approx E_{\infty(\Omega)} = \max_{i,j,k,l}{1}_{M^+}\left|u^e(x_j,y_k,z_l,t^i)-u^i_{j,k,l}\right|\label{eqn:bulk_infty}
\end{align}
\begin{align}
||u-u^i_h||_{L^2(\Omega)}\approx E_{L^2(\Omega)} = \max_{i}\left[\sum_{j,k,l}{1}_{M^+}\left(u(x_j,y_k,z_l,t^i)-u^i_{j,k,l}\right)^2h^3\right]^{\frac{1}{2}}\label{eqn:bulk_l2}
\end{align}
\begin{align}
&||u-u^i_h||_{H^1(\Omega)}\approx E_{H^1(\Omega)} =\max_{i}\Bigg[\sum_{j,k,l}{1}_{M^+}\left(u(x_j,y_k,z_l,t^i)-u^i_{j,k,l}\right)^2h^3\nonumber\\
&\quad+{1}_{M^+}\left(\frac{u(x_j+h,y_k,z_l,t^i)-u(x_j-h,y_k,z_l,t^i)}{2h}-\frac{u^i_{j+1,k,l}-u^i_{j-1,k,l}}{2h}\right)^2h^3\nonumber\\
&\quad+{1}_{M^+}\left(\frac{u(x_j,y_k+h,z_l,t^i)-u(x_j,y_k+h,z_l,t^i)}{2h}-\frac{u^i_{j,k+1,l}-u^i_{j,k-1,l}}{2h}\right)^2h^3\nonumber\\
&\quad+{1}_{M^+}\left(\frac{u(x_j,y_k,z_l+h,t^i)-u(x_j,y_k,z_l-h,t^i)}{2h}-\frac{u^i_{j,k,l+1}-u^i_{j,k,l-1}}{2h}\right)^2h^3\Bigg]^{\frac{1}{2}}\label{eqn:bulk_h1}
\end{align}
where $u^{i}_{j,k,l}\approx u(x_j,y_k,z_l,t^{i})$ and $u^i_h$ denotes also the numerical approximation to the exact solution at time $t^i$ using grid spacing $h$. Also, ${1}_{M^+}$ is the characteristic function for the point set $M^+$.

Additionally, we consider the $\infty$-norm error for the components in the gradient of the bulk solution $u^{i}_{j,k,l}$ at time level $t^i$. For example, the $\infty$-norm error of the $x$-component can be computed using the following formula:
\begin{align}\label{eqn:x-component}
 E_{\infty(\Omega)}=\max_{i,j,k,l}{1}_{M^+}\left|\frac{u(x_j+h,y_k,z_l,t^i)-u(x_j-h,y_k,z_l,t^i)}{2h}-\frac{u^i_{j+1,k,l}-u^i_{j-1,k,l}}{2h}\right|,
\end{align}
and the errors in $y,z$-components are computed similarly.

The approximations to the $\infty$-, $L^2$-norm and $H^1$-norm errors on the surface are computed using the following formulas respectively:
\begin{align}
||v-v^i_h||_{\infty(\Gamma)}\approx E_{\infty(\Gamma)} = \max_{i,j,k}\left|v(R,\theta_j,\varphi_k,t^i)-v^i_{j,k}\right|\label{eqn:surf_infty}
\end{align}

\begin{align}
||v-v^i_h||_{L^2(\Gamma)}\approx E_{L^2(\Gamma)} = \max_{i}\left[\sum_{j,k}\left(v(R,\theta_j,\varphi_k,t^i)-v^i_{j,k}\right)^2\sin\theta_j\Delta\theta\Delta\varphi\right]^{\frac{1}{2}}\label{eqn:surf_l2}
\end{align}

\begin{align}
&||v-v^i_h||_{H^1(\Gamma)}\approx E_{H^1(\Gamma)} = \max_{i}\Bigg[\sum_{j,k}\left(v(R,\theta_j,\varphi_k,t^i)-v^i_{j,k}\right)^2\sin\theta_j\Delta\theta\Delta\varphi\nonumber\\
&\quad+\left(\frac{v(R,\theta_j+\Delta\theta,\varphi_k,t^i)-v(R,\theta_j,\varphi_k,t^i)}{R\Delta\theta}-\frac{v^i_{j+1,k}-v^i_{j,k}}{R\Delta\theta}\right)^2\sin\theta_j\Delta\theta\Delta\varphi\nonumber\\
&\quad+\left(\frac{v(R,\theta_j,\varphi_k+\Delta\varphi,t^i)-v(R,\theta_j,\varphi_k,t^i)}{R\sin\theta_j\Delta\varphi}-\frac{v^i_{j,k+1}-v^i_{j,k}}{R\sin\theta_j\Delta\varphi}\right)^2\sin\theta_j\Delta\theta\Delta\varphi\Bigg]^{\frac{1}{2}}\label{eqn:surf_h1}
\end{align}
where $v^{i}_{j,k}\approx v(R,\theta_j,\varphi_k,t^{i})$ and $v^i_h$ denotes also the numerical approximation of the exact solution at time $t^i$. The increments in the discretization of $\theta$ and $\varphi$ are $\Delta\theta$ and $\Delta\varphi$ respectively. Moreover, in \eqref{eqn:surf_h1}, we require $\sin\theta_j\neq0$. For the surface errors of the model \eqref{eqn:dynamic-eq}--\eqref{eqn:dynamic-ic} with dynamic boundary condition, one simply replaces $v$ with $u$ in the formulas \eqref{eqn:surf_infty}--\eqref{eqn:surf_h1}.

Note that, for all the $\infty$-, $L^2$- and $H^1$-norm errors in space, the $\infty$-norm is taken in time.

%%%%%%%%%%%%%%%%%%%%%%%%%%%%%%%%%%%%%%%%%%%%%%%%%%%%%%%%%%%%%%%%%%%%%%%%%%%%%%%%%%%%%%%%%%%%%%%%%%%%%%%%%
\subsection{Dynamic Boundary Conditions}\label{sec:numerics-dbc}
In this subsection, we present the numerical results for models \eqref{eqn:dynamic-eq}--\eqref{eqn:dynamic-ic} with dynamic boundary conditions in a spherical domain with radius $R=0.5$.

\subsubsection{Test 1}

For the first test, we employ the exact solution $u(x,y,z,t)=e^t(x^2+2y^2+3z^2)$. The consideration of such a test problem is that it offers both simplicity and asymmetry in space.

\begin{table}
\centering
\footnotesize
\sisetup{
  output-exponent-marker = \text{ E},
  exponent-product={},
  retain-explicit-plus
}
\begin{tabular}{c S[table-format=1.4e2] c S[table-format=1.4e2] c S[table-format=1.4e2] c }
\toprule
{$N\times N\times N$} & {$E_{\infty(\Omega)}: u $} & Rate   & {$E_{L^2(\Omega)}: u$} & Rate & {$E_{H^1(\Omega)}: u$} & Rate\\
\midrule
$ 31 \times 31 \times 31$      & 5.7519E-06 & \textemdash & 3.4724E-06 & \textemdash & 4.7239E-06 & \textemdash\\
$ 63 \times 63 \times 63$      & 1.6449E-06 & 1.81        & 9.3307E-07 & 1.90        & 1.2730E-06 & 1.89\\
$ 127 \times 127 \times 127$   & 4.0469E-07 & 2.02        & 2.3127E-07 & 2.01        & 3.1149E-07 & 2.03\\
$ 255 \times 255 \times 255$   & 1.0445E-07 & 1.95        & 5.8647E-08 & 1.98        & 7.9459E-08 & 1.97\\
\midrule
{$N\times N\times N$} & {$E_{\infty(\Gamma)}: u$} & Rate   & {$E_{L^2(\Gamma)}: u$} & Rate & {$E_{H^1(\Gamma)}: u$} & Rate\\
\midrule
$ 31 \times 31 \times 31$      & 5.8021E-06 & \textemdash & 8.9307E-06 & \textemdash & 9.6440E-06 & \textemdash\\
$ 63 \times 63 \times 63$      & 1.6613E-06 & 1.80        & 2.4687E-06 & 1.86        & 2.7621E-06 & 1.80\\
$ 127 \times 127 \times 127$   & 4.0576E-07 & 2.03        & 6.1087E-07 & 2.01        & 6.7467E-07 & 2.03\\
$ 255 \times 255 \times 255$   & 1.0469E-07 & 1.95        & 1.5629E-07 & 1.97        & 1.7384E-07 & 1.96\\
\midrule
{$N\times N\times N$} & {$E_{\infty(\Omega)}:\nabla_xu$} & Rate   & {$E_{\infty(\Omega)}:\nabla_y u$} & Rate & {$E_{\infty(\Omega)}:\nabla_z u$} & Rate\\
\midrule
$ 31 \times 31 \times 31$      & 1.8856E-06 & \textemdash & 4.6662E-06 & \textemdash & 7.7164E-06 & \textemdash\\
$ 63 \times 63 \times 63$      & 4.1092E-07 & 2.20        & 1.1122E-06 & 2.07        & 1.8944E-06 & 2.03\\
$ 127 \times 127 \times 127$   & 9.8629E-08 & 2.06        & 2.7337E-07 & 2.02        & 4.6526E-07 & 2.03\\
$ 255 \times 255 \times 255$   & 2.4337E-08 & 2.02        & 6.7944E-08 & 2.01        & 1.1611E-07 & 2.00\\
\bottomrule
\end{tabular}
\caption{Convergence of the $\infty$-, $L^2$- and $H^1$-norm errors of the solutions in the bulk/surface, and the ${\infty}$-norm errors of the gradients in the bulk for the dynamic BC model \eqref{eqn:dynamic-eq}--\eqref{eqn:dynamic-ic} with exact solution $u=e^t(x^2+2y^2+3z^2)$ until final time $T=0.1$ in the sphere of $R=0.5$. The number of spherical harmonics for term $u$ is 9.}
\label{table:dynamicBC-T1}
\end{table}

In Table~\ref{table:dynamicBC-T1}, we observe that in the bulk, the $L^2$-norm errors are smaller than the $\infty$ errors, which is as expected. However, on the surface, the $L^2$-norm errors are larger than the $\infty$-norm errors. This can be explained by the following estimate of the $L^2$-norm errors:
\begin{align}
E_{L^2(\Gamma)}&=\max_i\left[\sum_{j,k}\left(u(R,\theta_j,\varphi_k,t^i)-u^i_{j,k}\right)^2\sin\theta_j\Delta\theta\Delta\varphi\right]^{\frac{1}{2}}\label{eqn:inequality1}\\
&\leq \max_{i,j,k}\left|u(R,\theta_j,\varphi_k,t^i)-u^i_{j,k}\right|\left[\sum_{j,k}\sin\theta_j\Delta\theta\Delta\varphi\right]^{\frac{1}{2}}\\
&\approx\sqrt{4\pi R^2}\max_{i,j,k}\left|u(R,\theta_j,\varphi_k,t^i)-u^i_{j,k}\right|\\
&=2R\sqrt{\pi}E_{{\infty}(\Gamma)}\label{eqn:inequality4}
\end{align}
We observe the overall second-order convergence in all norms of the
errors for solutions, both on the surface and in the bulk. Note that, the $\infty$-norm errors of the gradients in the bulk also obey the second-order convergence, as well as the $H^1$-norm errors in the bulk and on the surface.

%%%%%%%%%%%%%%%%%%%%%%%%%%%%%%%%%%%%%%%%%%%%%%%%%%%%%%%%%%%%%%%%%%%%%%%%%
\begin{table}
\centering
\footnotesize
\sisetup{
  output-exponent-marker = \text{ E},
  exponent-product={},
  retain-explicit-plus
}
\begin{tabular}{c S[table-format=1.4e2] c S[table-format=1.4e2] c S[table-format=1.4e2] c }
\toprule
{$N\times N\times N$} & {$E_{\infty(\Omega)}: u $} & Rate   & {$E_{L^2(\Omega)}: u$} & Rate & {$E_{H^1(\Omega)}: u$} & Rate\\
\midrule
$ 31 \times 31 \times 31$      & 1.6347E-04 & \textemdash & 4.8011E-05 & \textemdash & 2.9914E-04 & \textemdash\\
$ 63 \times 63 \times 63$      & 2.4156E-05 & 2.76        & 6.4637E-06 & 2.89        & 4.9519E-05 & 2.59\\
$ 127 \times 127 \times 127$   & 3.4620E-06 & 2.80        & 9.1209E-07 & 2.83        & 8.3727E-06 & 2.56\\
$ 255 \times 255 \times 255$   & 5.5230E-07 & 2.65        & 1.4151E-07 & 2.69        & 1.4736E-06 & 2.51\\
\midrule
{$N\times N\times N$} & {$E_{\infty(\Gamma)}: u$} & Rate   & {$E_{L^2(\Gamma)}: u$} & Rate & {$E_{H^1(\Gamma)}: u$} & Rate\\
\midrule
$ 31 \times 31 \times 31$      & 1.6052E-05 & \textemdash & 1.1379E-05 & \textemdash & 3.3818E-05 & \textemdash\\
$ 63 \times 63 \times 63$      & 3.0297E-06 & 2.41        & 3.0092E-06 & 1.92        & 5.5188E-06 & 2.62\\
$ 127 \times 127 \times 127$   & 4.7815E-07 & 2.66        & 6.9355E-07 & 2.12        & 8.9181E-07 & 2.63\\
$ 255 \times 255 \times 255$   & 1.1105E-07 & 2.11        & 1.6069E-07 & 2.11        & 1.7645E-07 & 1.96\\
\midrule
{$N\times N\times N$} & {$E_{\infty(\Omega)}:\nabla_xu$} & Rate   & {$E_{\infty(\Omega)}:\nabla_y u$} & Rate & {$E_{\infty(\Omega)}:\nabla_z u$} & Rate\\
\midrule
$ 31 \times 31 \times 31$      & 1.8435E-03 & \textemdash & 1.8622E-03 & \textemdash & 1.8753E-03 & \textemdash\\
$ 63 \times 63 \times 63$      & 4.8836E-04 & 1.92        & 4.5222E-04 & 2.04        & 4.6395E-04 & 2.02\\
$ 127 \times 127 \times 127$   & 1.2831E-04 & 1.93        & 1.1405E-04 & 1.99        & 1.0923E-04 & 2.09\\
$ 255 \times 255 \times 255$   & 3.4607E-05 & 1.89        & 3.2892E-05 & 1.79        & 3.1182E-05 & 1.81\\
\bottomrule
\end{tabular}
\caption{Convergence: $d$ perturbed by $\epsilon h^3$ for for the dynamic BC model \eqref{eqn:dynamic-eq}--\eqref{eqn:dynamic-ic} with exact solution $u=e^t(x^2+2y^2+3z^2)$ until final time $T=0.1$ in the sphere of $R=0.5$. The number of spherical harmonics for term $u$ is 9.}
\label{table:d-perturbed}
\end{table}

%%%%%%%%%%%%%%%%%%%%%%%%%%%%%%%%%%%%%%%%%%%%%%%%%%%%%%%%%%%%%%%%%%%%%%%

\begin{table}
\centering
\footnotesize
\sisetup{
  output-exponent-marker = \text{ E},
  exponent-product={},
  retain-explicit-plus
}
\begin{tabular}{c S[table-format=1.4e2] c S[table-format=1.4e2] c S[table-format=1.4e2] c }
\toprule
{$N\times N\times N$} & {$E_{\infty(\Omega)}: u $} & Rate   & {$E_{L^2(\Omega)}: u$} & Rate & {$E_{H^1(\Omega)}: u$} & Rate\\
\midrule
$ 31 \times 31 \times 31$      & 3.1314E-05 & \textemdash & 5.5099E-06 & \textemdash & 5.8781E-05 & \textemdash\\
$ 63 \times 63 \times 63$      & 4.9768E-06 & 2.65        & 1.0797E-06 & 2.35        & 1.1369E-05 & 2.37\\
$ 127 \times 127 \times 127$   & 9.7584E-07 & 2.35        & 2.3402E-07 & 2.21        & 2.2574E-06 & 2.33\\
$ 255 \times 255 \times 255$   & 1.8121E-07 & 2.43        & 5.9103E-08 & 1.99        & 4.7119E-07 & 2.26\\
\midrule
{$N\times N\times N$} & {$E_{\infty(\Gamma)}: u$} & Rate   & {$E_{L^2(\Gamma)}: u$} & Rate & {$E_{H^1(\Gamma)}: u$} & Rate\\
\midrule
$ 31 \times 31 \times 31$      & 3.0261E-05 & \textemdash & 2.6860E-05 & \textemdash & 7.6184E-05 & \textemdash\\
$ 63 \times 63 \times 63$      & 4.7251E-06 & 2.68        & 4.0207E-06 & 2.74        & 9.6476E-06 & 2.98\\
$ 127 \times 127 \times 127$   & 7.7908E-07 & 2.61        & 7.1346E-07 & 2.49        & 1.3434E-06 & 2.84\\
$ 255 \times 255 \times 255$   & 1.5157E-07 & 2.36        & 1.6333E-07 & 2.13        & 2.2525E-07 & 2.58\\
\midrule
{$N\times N\times N$} & {$E_{\infty(\Omega)}:\nabla_xu$} & Rate   & {$E_{\infty(\Omega)}:\nabla_y u$} & Rate & {$E_{\infty(\Omega)}:\nabla_z u$} & Rate\\
\midrule
$ 31 \times 31 \times 31$      & 4.7534E-04 & \textemdash & 3.6962E-04 & \textemdash & 4.0036E-04 & \textemdash\\
$ 63 \times 63 \times 63$      & 1.3972E-04 & 1.77        & 1.0855E-04 & 1.77        & 1.2844E-04 & 1.64\\
$ 127 \times 127 \times 127$   & 3.7348E-05 & 1.90        & 2.8852E-05 & 1.91        & 3.4981E-05 & 1.88\\
$ 255 \times 255 \times 255$   & 1.0169E-05 & 1.88        & 7.8385E-06 & 1.88        & 9.6735E-06 & 1.85\\
\bottomrule
\end{tabular}
\caption{Convergence: $\theta$ perturbed by $\epsilon h^3$ for for the dynamic BC model \eqref{eqn:dynamic-eq}--\eqref{eqn:dynamic-ic} with exact solution $u=e^t(x^2+2y^2+3z^2)$ until final time $T=0.1$ in the sphere of $R=0.5$. The number of spherical harmonics for term $u$ is 9.}\label{table:theta-perturbed}
\end{table}

%%%%%%%%%%%%%%%%%%%%%%%%%%%%%%%%%%%%%%%%%%%%%%%%%%%%%%%%%%%%%%%%%%%
\begin{table}
\centering
\footnotesize
\sisetup{
  output-exponent-marker = \text{ E},
  exponent-product={},
  retain-explicit-plus
}
\begin{tabular}{c S[table-format=1.4e2] c S[table-format=1.4e2] c S[table-format=1.4e2] c }
\toprule
{$N\times N\times N$} & {$E_{\infty(\Omega)}: u $} & Rate   & {$E_{L^2(\Omega)}: u$} & Rate & {$E_{H^1(\Omega)}: u$} & Rate\\
\midrule
$ 31 \times 31 \times 31$      & 1.2008E-05 & \textemdash & 3.5593E-006 & \textemdash & 2.1347E-005 & \textemdash\\
$ 63 \times 63 \times 63$      & 2.2351E-06 & 2.43        & 9.5723E-007 & 2.89        & 4.1701E-006 & 2.59\\
$ 127 \times 127 \times 127$   & 4.8683E-07 & 2.20        & 2.3412E-007 & 2.83        & 7.9991E-007 & 2.56\\
$ 255 \times 255 \times 255$   & 1.0523E-07 & 2.21        & 5.8696E-008 & 2.69        & 1.5425E-007 & 2.51\\
\midrule
{$N\times N\times N$} & {$E_{\infty(\Gamma)}: u$} & Rate   & {$E_{L^2(\Gamma)}: u$} & Rate & {$E_{H^1(\Gamma)}: u$} & Rate\\
\midrule
$ 31 \times 31 \times 31$      & 1.3088E-05 & \textemdash & 1.1687E-005 & \textemdash & 3.8782E-005 & \textemdash\\
$ 63 \times 63 \times 63$      & 2.2948E-06 & 2.51        & 2.6848E-006 & 1.92        & 5.2359E-006 & 2.62\\
$ 127 \times 127 \times 127$   & 4.4200E-07 & 2.38        & 6.2695E-007 & 2.12        & 8.5292E-007 & 2.63\\
$ 255 \times 255 \times 255$   & 1.0517E-07 & 2.07        & 1.5687E-007 & 2.11        & 1.8487E-007 & 1.96\\
\midrule
{$N\times N\times N$} & {$E_{\infty(\Omega)}:\nabla_xu$} & Rate   & {$E_{\infty(\Omega)}:\nabla_y u$} & Rate & {$E_{\infty(\Omega)}:\nabla_z u$} & Rate\\
\midrule
$ 31 \times 31 \times 31$      & 1.6985E-004 & \textemdash & 1.6487E-004 & \textemdash & 1.4704E-004 & \textemdash\\
$ 63 \times 63 \times 63$      & 4.8808E-005 & 1.92        & 5.2182E-005 & 2.04        & 4.5240E-005 & 2.02\\
$ 127 \times 127 \times 127$   & 1.4275E-005 & 1.93        & 1.4329E-005 & 1.99        & 1.2582E-005 & 2.09\\
$ 255 \times 255 \times 255$   & 3.9625E-006 & 1.89        & 3.9751E-006 & 1.79        & 3.6156E-006 & 1.81\\
\bottomrule
\end{tabular}
\caption{Convergence: $\varphi$ perturbed by $\epsilon h^3$ for for the dynamic BC model \eqref{eqn:dynamic-eq}--\eqref{eqn:dynamic-ic} with exact solution $u=e^t(x^2+2y^2+3z^2)$ until final time $T=0.1$ in the sphere of $R=0.5$. The number of spherical harmonics for term $u$ is 9.}
\label{table:phi-perturbed}
\end{table}

\begin{table}
\centering
\footnotesize
\sisetup{
  output-exponent-marker = \text{ E},
  exponent-product={},
  retain-explicit-plus
}
\begin{tabular}{c S[table-format=1.4e2] c S[table-format=1.4e2] c S[table-format=1.4e2] c }
\toprule
{$N\times N\times N$} & {$E_{\infty(\Omega)}: u $} & Rate   & {$E_{L^2(\Omega)}: u$} & Rate & {$E_{H^1(\Omega)}: u$} & Rate\\
\midrule
$ 31 \times 31 \times 31$      & 1.6943E-004 & \textemdash & 4.8339E-005 & \textemdash & 3.0187E-004 & \textemdash\\
$ 63 \times 63 \times 63$      & 2.4949E-005 & 2.76        & 6.5310E-006 & 2.89        & 5.0688E-005 & 2.59\\
$ 127 \times 127 \times 127$   & 3.4388E-006 & 2.80        & 9.1301E-007 & 2.83        & 8.6516E-006 & 2.56\\
$ 255 \times 255 \times 255$   & 5.5585E-007 & 2.65        & 1.4177E-007 & 2.69        & 1.5262E-006 & 2.51\\
\midrule
{$N\times N\times N$} & {$E_{\infty(\Gamma)}: u$} & Rate   & {$E_{L^2(\Gamma)}: u$} & Rate & {$E_{H^1(\Gamma)}: u$} & Rate\\
\midrule
$ 31 \times 31 \times 31$      & 3.0730E-005 & \textemdash & 2.6648E-005 & \textemdash & 8.3318E-005 & \textemdash\\
$ 63 \times 63 \times 63$      & 4.4242E-006 & 2.41        & 3.8308E-006 & 1.92        & 8.5970E-006 & 2.62\\
$ 127 \times 127 \times 127$   & 8.4725E-007 & 2.66        & 8.1565E-007 & 2.12        & 1.5920E-006 & 2.63\\
$ 255 \times 255 \times 255$   & 1.4816E-007 & 2.11        & 1.6760E-007 & 2.11        & 2.3146E-007 & 1.96\\
\midrule
{$N\times N\times N$} & {$E_{\infty(\Omega)}:\nabla_xu$} & Rate   & {$E_{\infty(\Omega)}:\nabla_y u$} & Rate & {$E_{\infty(\Omega)}:\nabla_z u$} & Rate\\
\midrule
$ 31 \times 31 \times 31$      & 2.0305E-003 & \textemdash & 1.7989E-003 & \textemdash & 1.7563E-003 & \textemdash\\
$ 63 \times 63 \times 63$      & 5.5069E-004 & 1.92        & 5.0127E-004 & 2.04        & 5.0564E-004 & 2.02\\
$ 127 \times 127 \times 127$   & 1.3554E-004 & 1.93        & 1.2851E-004 & 1.99        & 1.2397E-004 & 2.09\\
$ 255 \times 255 \times 255$   & 3.9576E-005 & 1.89        & 3.5780E-005 & 1.79        & 3.2330E-005 & 1.81\\
\bottomrule
\end{tabular}
\caption{Convergence: $d,\theta,\varphi$ perturbed by $\epsilon h^3$ for for the dynamic BC model \eqref{eqn:dynamic-eq}--\eqref{eqn:dynamic-ic} with exact solution $u=e^t(x^2+2y^2+3z^2)$ until final time $T=0.1$ in the sphere of $R=0.5$. The number of spherical harmonics for term $u$ is 9.}
\label{table:all-perturbed}
\end{table}

In Tables \ref{table:d-perturbed}--\ref{table:all-perturbed}, we present the convergence results for the dynamic BC model \eqref{eqn:dynamic-eq}--\eqref{eqn:dynamic-ic} with perturbed $d$, $\theta$, and $\varphi$ in the extension operator \eqref{eqn:extension_operator}. (The perturbations in $\theta$ and $\varphi$ mimic the ``errors'' in the normal vector $n$.) We investigated numerically with different choices of perturbations and present the results with $\epsilon h^3$ ($\epsilon$ is a pseudo-random number sampled uniformly from $[0,1]$), which preserves the second-order accuracy of the solution in  $\infty$-, $L^2$- and $H^1$-norm, and the gradient components in $\infty$-norm. Since the random perturbation is added at every point in $\gamma$ set, the total perturbation is in the order of $\mathcal{O}(h)$.
Note that tests with perturbations in Tables \ref{table:d-perturbed}--\ref{table:all-perturbed} illustrate that the proposed DPM-based algorithm preserves the second-order accuracy even in situations where the signed distances and the normal vectors to the surface boundary are not known exactly.

\begin{figure}
    \centering
    \begin{subfigure}[b]{0.49\textwidth}
        \includegraphics[width=\textwidth]{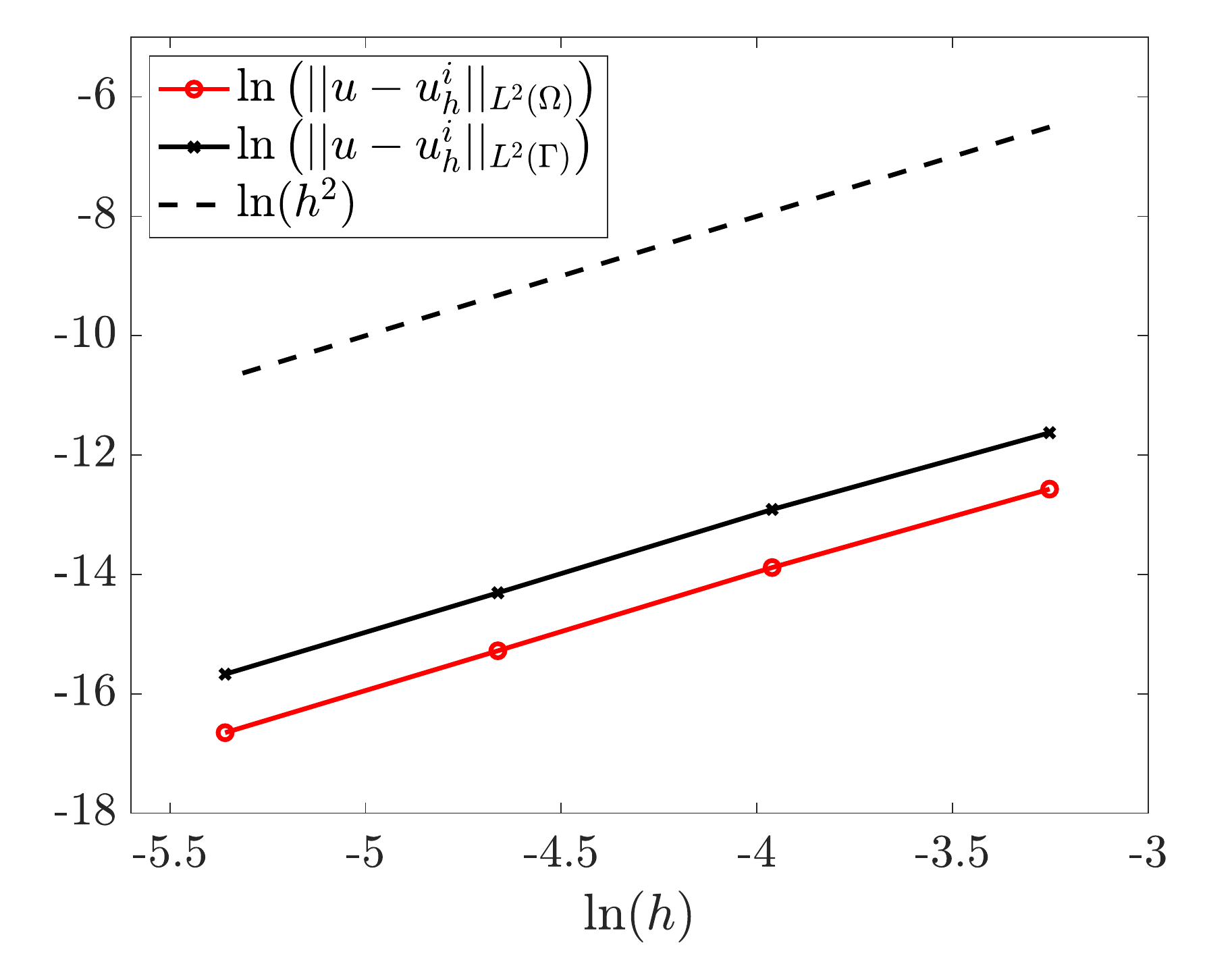}
        \caption{}
    \end{subfigure}
    \begin{subfigure}[b]{0.49\textwidth}
        \includegraphics[width=\textwidth]{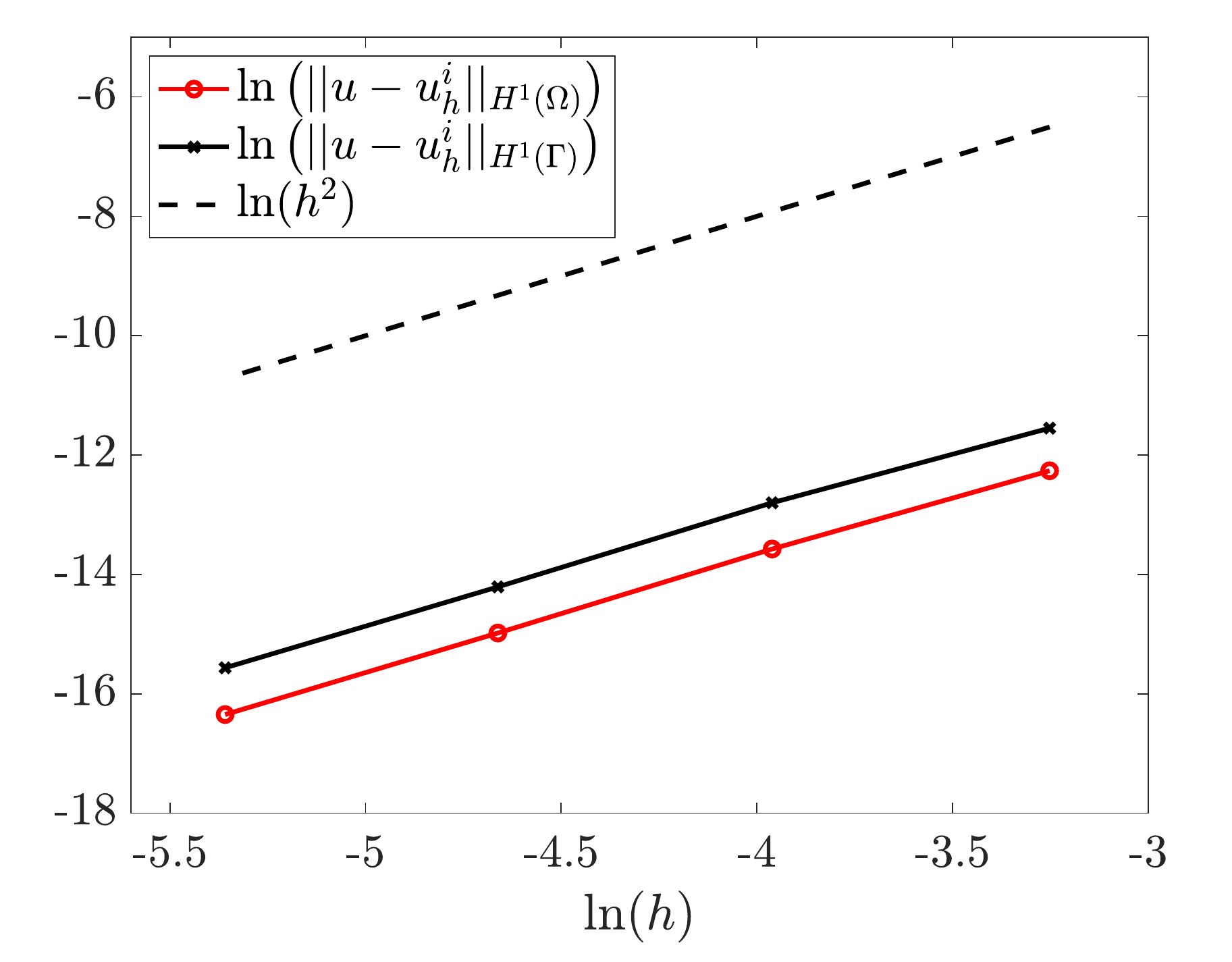}
        \caption{}
    \end{subfigure}
    \caption{Log-log plots of bulk/surface $L^2$-norm errors (left
      figure) and bulk/surface $H^1$-norm errors (right figure) for
      the dynamic BC model
      \eqref{eqn:dynamic-eq}--\eqref{eqn:dynamic-ic} with the exact solution $u=e^t(x^2+2y^2+3z^2)$ in the sphere of $R=0.5$.}
    \label{fig:dynamicBC-T1-errors}
\end{figure}

\begin{figure}
    \centering
    \begin{subfigure}[b]{0.7\textwidth}
        \includegraphics[width=\textwidth]{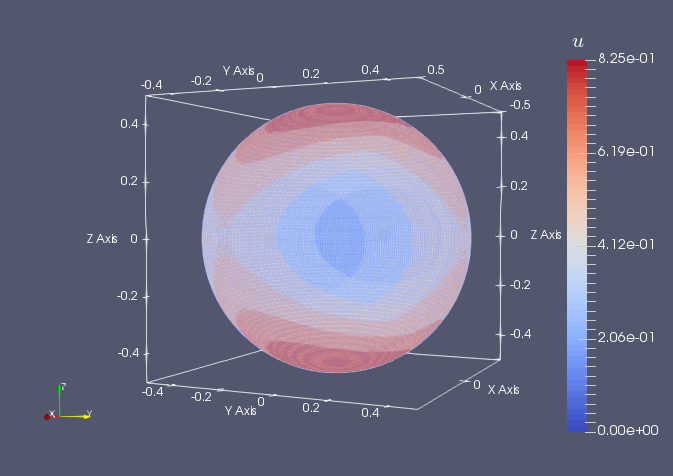}
        \caption{}
    \end{subfigure}
    \begin{subfigure}[b]{0.7\textwidth}
        \includegraphics[width=\textwidth]{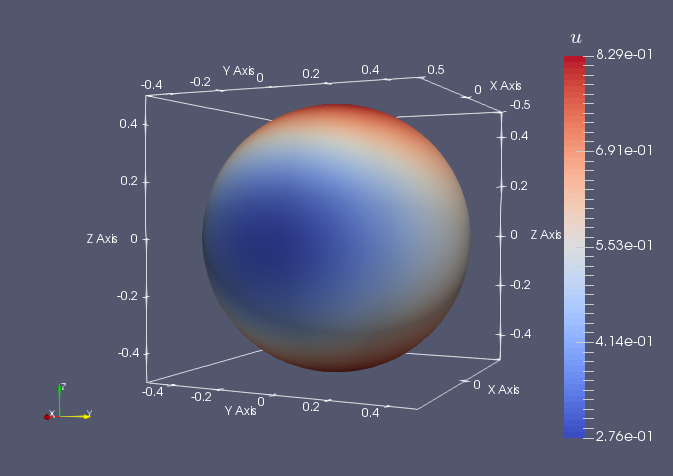}
        \caption{}
    \end{subfigure}
    \caption{3D views of the bulk (top figure) and the
      surface (bottom figure) approximations from mesh
      $255\times255\times255$ at $T=0.1$ to the dynamic BC model
      \eqref{eqn:dynamic-eq}--\eqref{eqn:dynamic-ic} with the exact solution $u=e^t(x^2+2y^2+3z^2)$ in the sphere of $R=0.5$.}
    \label{fig:dynamicBC-T1-bulk-surf-3dview}
\end{figure}

In Fig.~\ref{fig:dynamicBC-T1-errors}, we observe that the
  behavior of the $L^2$- and $H^1$-norm errors are very similar in
the bulk and on the surface. Besides, the errors in the bulk are
smaller than the errors on the surface in both norms. Also,
as already mentioned, the $H^1$-norm errors give the
second-order convergence, as opposed to the results obtained, for
example, from the finite element method, e.g.,
\cite{Hansbo_2016}. Moreover, the errors are far below the reference
dashed lines, which implies a small error constant  in \eqref{eqn:accuracy}.
In Fig. \ref{fig:dynamicBC-T1-bulk-surf-3dview}, we show the 3D isosurface plots (analogous to the contour plots in 2D) in the top figure and the plot of surface solution in the bottom figure, obtained using mesh $255\times255\times255$ at the final time $T=0.1$. 

\subsubsection{Test 2}
In this subsection, we use the exact solution $u(x,y,z,t)=e^t\sin(x)\sin(2y)\sin(3z)$. Compared to the first test, this choice of test is more oscillatory and requires a larger number of spherical harmonics to resolve $u$ and $u_{rr}$ accurately on the boundary $\Gamma$. Nevertheless, the total number of harmonics is still much less than $|\gamma_{in}|$, see Table~\ref{table:dynamicBC-T2}, for example.

\begin{table}
\centering
\footnotesize
\sisetup{
  output-exponent-marker = \text{ E},
  exponent-product={},
  retain-explicit-plus
}
\begin{tabular}{c S[table-format=1.4e2] c S[table-format=1.4e2] c S[table-format=1.4e2] c }
\toprule
{$N\times N\times N$} & {$E_{\infty(\Omega)}: u $} & Rate   & {$E_{L^2(\Omega)}: u$} & Rate & {$E_{H^1(\Omega)}: u$} & Rate\\
\midrule
$ 31 \times 31 \times 31$      & 8.7147E-06 & \textemdash & 2.4039E-06 & \textemdash & 2.8357E-05 & \textemdash\\
$ 63 \times 63 \times 63$      & 1.7811E-06 & 2.29        & 5.3478E-07 & 2.17        & 5.6563E-06 & 2.33\\
$ 127 \times 127 \times 127$   & 4.3585E-07 & 2.03        & 1.3077E-07 & 2.03        & 1.3353E-06 & 2.08\\
$ 255 \times 255 \times 255$   & 1.0849E-07 & 2.01        & 3.2659E-08 & 2.00        & 3.3032E-07 & 2.02\\
\midrule
{$N\times N\times N$} & {$E_{\infty(\Gamma)}: u$} & Rate   & {$E_{L^2(\Gamma)}: u$} & Rate & {$E_{H^1(\Gamma)}: u$} & Rate\\
\midrule
$ 31 \times 31 \times 31$      & 5.8527E-06 & \textemdash & 5.1084E-06 & \textemdash & 3.7771E-05 & \textemdash\\
$ 63 \times 63 \times 63$      & 1.4134E-06 & 2.05        & 1.2405E-06 & 2.04        & 9.0795E-06 & 2.06\\
$ 127 \times 127 \times 127$   & 3.4714E-07 & 2.03        & 3.0530E-07 & 2.02        & 2.2280E-06 & 2.03\\
$ 255 \times 255 \times 255$   & 8.5729E-08 & 2.02        & 7.5389E-08 & 2.02        & 5.5028E-07 & 2.02\\
\midrule
{$N\times N\times N$} & {$E_{\infty(\Omega)}:\nabla_xu$} & Rate   & {$E_{\infty(\Omega)}:\nabla_y u$} & Rate & {$E_{\infty(\Omega)}:\nabla_z u$} & Rate\\
\midrule
$ 31 \times 31 \times 31$      & 8.5641E-05 & \textemdash & 9.0012E-05 & \textemdash & 1.6055E-04 & \textemdash\\
$ 63 \times 63 \times 63$      & 2.3716E-05 & 1.85        & 2.0313E-05 & 2.15        & 5.1744E-05 & 1.63\\
$ 127 \times 127 \times 127$   & 5.0683E-06 & 2.27        & 5.2314E-06 & 1.96        & 1.1111E-05 & 2.22\\
$ 255 \times 255 \times 255$   & 1.3229E-06 & 1.94        & 1.1629E-06 & 2.17        & 2.8860E-06 & 1.94\\
\bottomrule
\end{tabular}
\caption{Convergence of the $\infty$-, $L^2$- and $H^1$-norm errors of the solutions in the bulk/surface, and the ${\infty}$-norm errors of gradients in the bulk for the dynamic BC model \eqref{eqn:dynamic-eq}--\eqref{eqn:dynamic-ic} with the exact solution $u=e^t\sin(x)\sin(2y)\sin(3z)$ until final time $T=0.1$ in the sphere of $R=0.5$. The number of spherical harmonics for terms $u$ is 400 per each term.}
\label{table:dynamicBC-T2}
\end{table}

\begin{figure}
    \centering
    \begin{subfigure}[b]{0.49\textwidth}
        \includegraphics[width=\textwidth]{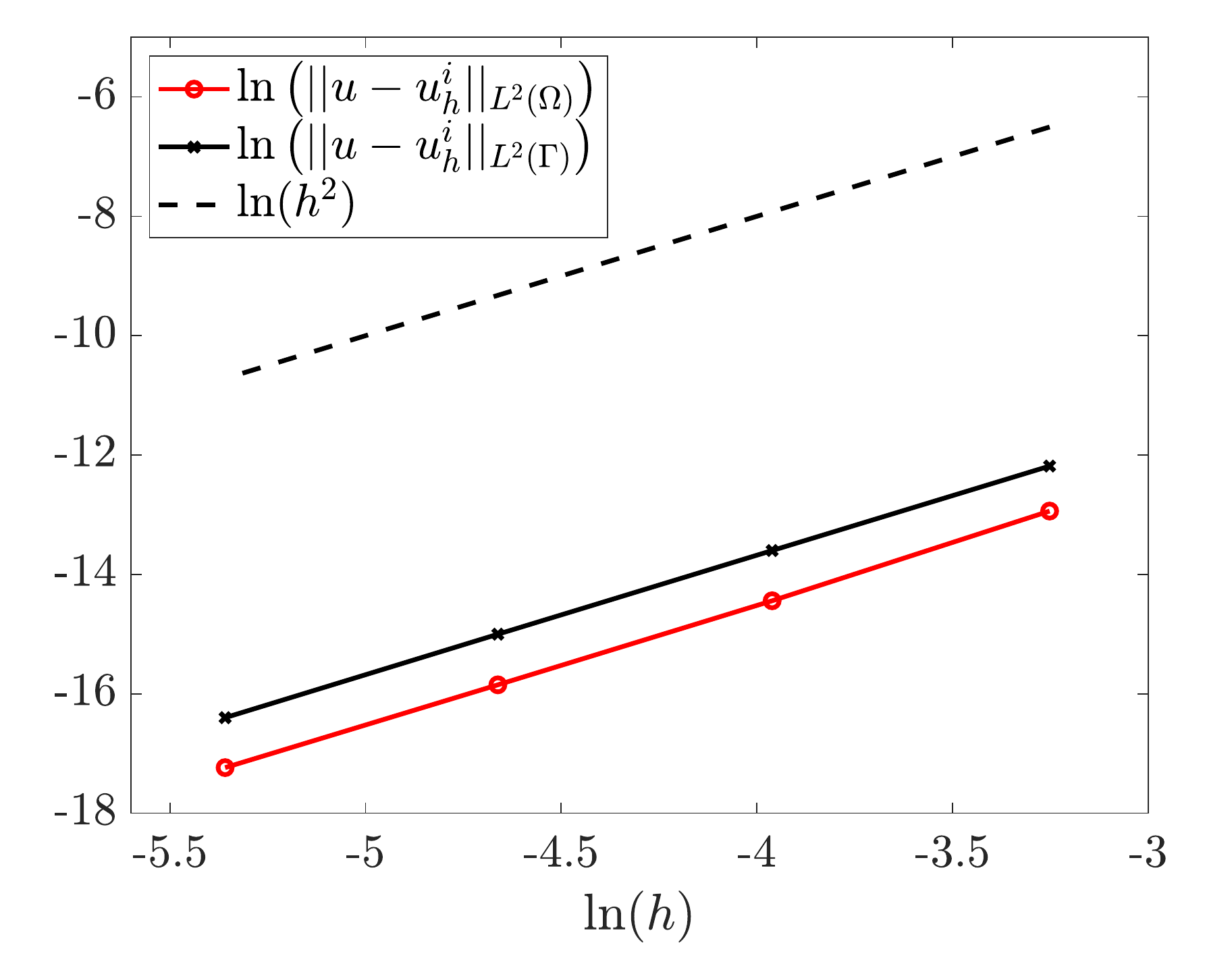}
        \caption{}
    \end{subfigure}
    \begin{subfigure}[b]{0.49\textwidth}
        \includegraphics[width=\textwidth]{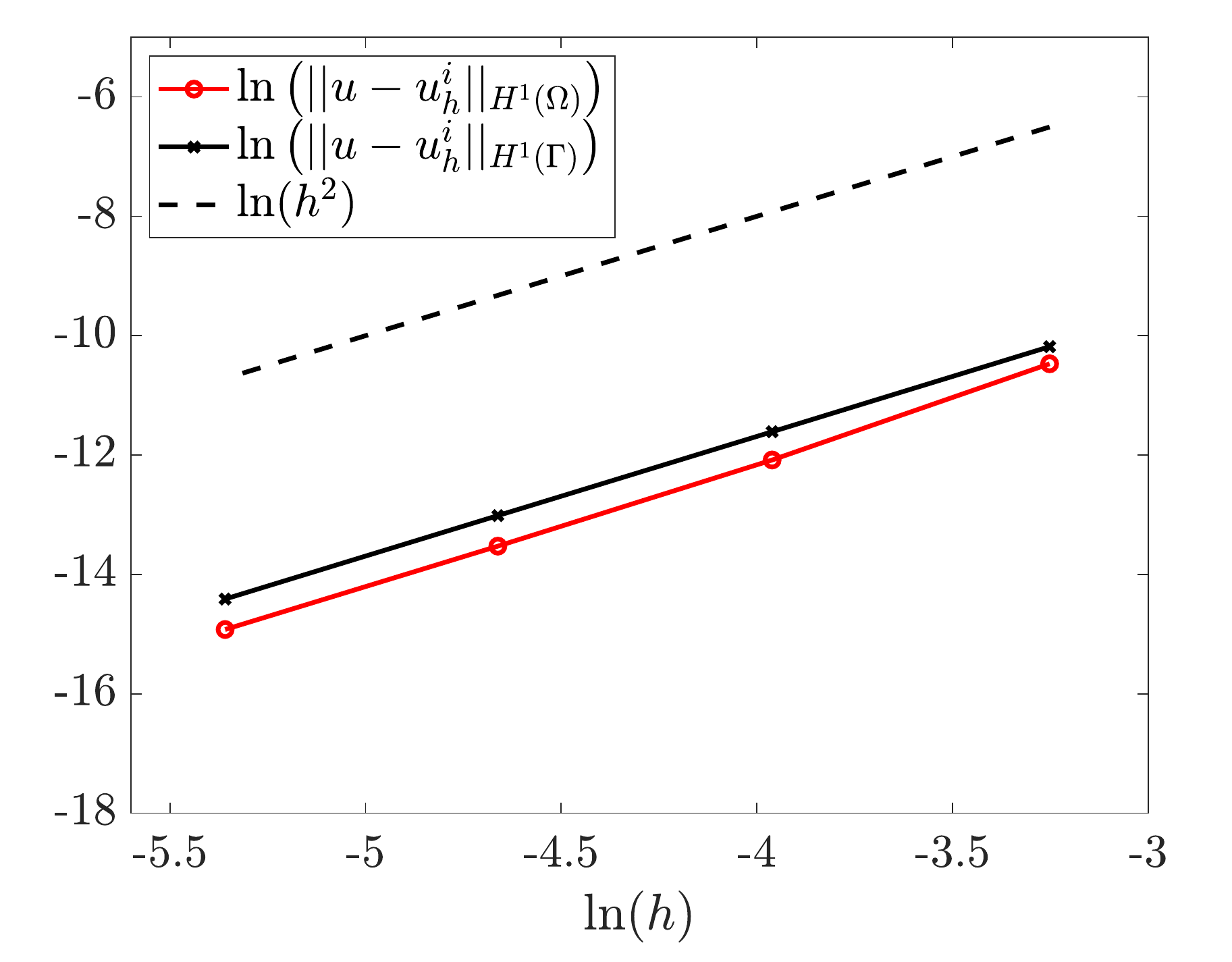}
        \caption{}
    \end{subfigure}
    \caption{Log-log plots of bulk/surface $L^2$-norm errors (left
      figure) and bulk/surface $H^1$-norm errors (right figure) for
      the dynamic BC model
      \eqref{eqn:dynamic-eq}--\eqref{eqn:dynamic-ic} with the exact solution $u=e^t\sin(x)\sin(2y)\sin(3z)$ in the sphere of $R=0.5$.}
    \label{fig:dynamicBC-T2-errors}
\end{figure}

\begin{figure}
    \centering
    \begin{subfigure}[b]{0.7\textwidth}
        \includegraphics[width=\textwidth]{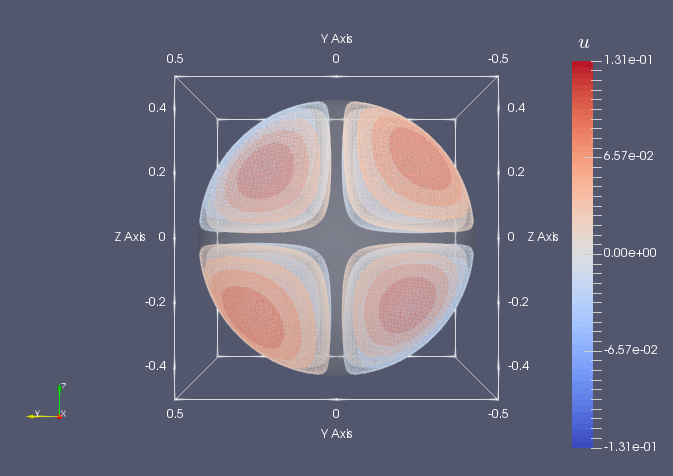}
        \caption{}
    \end{subfigure}
    \begin{subfigure}[b]{0.7\textwidth}
        \includegraphics[width=\textwidth]{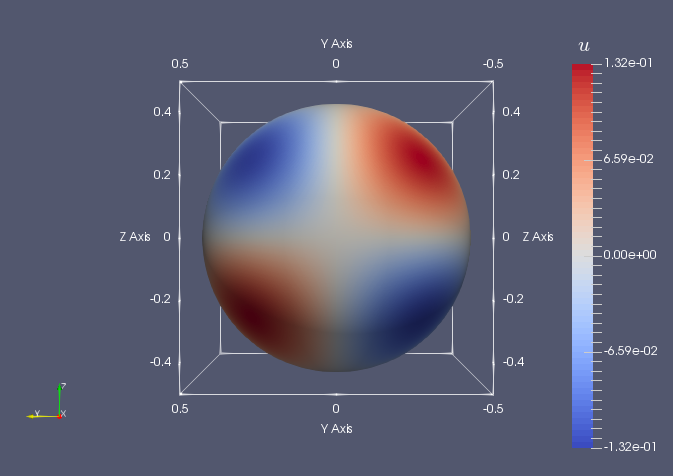}
        \caption{}
    \end{subfigure}
    \caption{3D views of the bulk (top figure) and the
      surface (bottom figure) approximations from mesh
      $255\times255\times255$ at $T=0.1$ to the dynamic BC model
      \eqref{eqn:dynamic-eq}--\eqref{eqn:dynamic-ic} with the exact solution $u=e^t\sin(x)\sin(2y)\sin(3z)$ in the sphere of $R=0.5$. }
    \label{fig:dynamicBC-T2-bulk-surf-3dview}
\end{figure} 

In Table~\ref{table:dynamicBC-T2}, again we observe second-order accuracy in all norms of the solutions in the bulk and on the surface.

Similarly, in Fig. \ref{fig:dynamicBC-T2-errors}, the errors in the bulk are smaller than the errors on the surface. In Fig. \ref{fig:dynamicBC-T2-bulk-surf-3dview}, we give the 3D isosurface plots in the top figure and the plot of surface solution in the bottom figure, obtained using mesh $255\times255\times255$ at final time $T=0.1$.

%%%%%%%%%%%%%%%%%%%%%%%%%%%%%%%%%%%%%%%%%%%%%%%%%%%%%%%%%%%%%%%%%%%
\subsection{Linear Bulk-Surface Coupling}
In this subsection, we present the numerical results for the model
\eqref{eqn:bulk}--\eqref{eqn:surface-ic}, with linear bulk-surface
coupling, i.e., $h(u,v)=u-v$ in a spherical domain of radius $R=1$. In
particular, the exact solutions
$u(x,y,z,t)=e^te^{-x(x-1)-y(y-1)}$ and $v(x,y,z,t) =
e^te^{-x(x-1)-y(y-1)}(1+x(1-2x)+y(1-2y))$ are such that the coupling
condition \eqref{eqn:coupling} is satisfied exactly on the surface
(the test is modification of the tests from \cite{Burman_2015,Elliott_2012}). Additionally, we provide numerical results to compare with the ones obtained using the cut finite element method in \cite{Burman_2015}.

\begin{remark}
We should note that the comparisons between the DPM-based
  method in this work and the cut-FEM approach in \cite{Burman_2015}
  are not precise, since the exact solutions
  $u(x,y,z)=e^{-x(x-1)-y(y-1)}$ and $v(x,y,z)
  =e^{-x(x-1)-y(y-1)}(1+x(1-2x)+y(1-2y))$ in \cite{Burman_2015}
are considered for the elliptic type bulk-surface problems. Nevertheless, we add $e^{t}$ in the exact solutions and take the $\infty$-norm errors in time, in the hope to discuss the difference and similarity between the two approaches.
\end{remark}

\begin{table}
\centering
\footnotesize
\sisetup{
  output-exponent-marker = \text{ E},
  exponent-product={},
  retain-explicit-plus
}
\begin{tabular}{c S[table-format=1.4e2] c S[table-format=1.4e2] c S[table-format=1.4e2] c }
\toprule
{$N\times N\times N$} & {$E_{\infty(\Omega)}: u $} & Rate   & {$E_{L^2(\Omega)}: u$} & Rate & {$E_{H^1(\Omega)}: u$} & Rate\\
\midrule
$ 31 \times 31 \times 31$      & 1.2537E-03 & \textemdash & 9.5344E-04 & \textemdash & 3.5388E-03 & \textemdash\\
$ 63 \times 63 \times 63$      & 2.9791E-04 & 2.07        & 2.2803E-04 & 2.06        & 7.0225E-04 & 2.33\\
$ 127 \times 127 \times 127$   & 7.2333E-05 & 2.04        & 5.5188E-05 & 2.05        & 1.7337E-04 & 2.02\\
$ 255 \times 255 \times 255$   & 1.7734E-05 & 2.03        & 1.3555E-05 & 2.03        & 4.2658E-05 & 2.02\\
\midrule
{$N\times N\times N$} & {$E_{\infty(\Gamma)}: v$} & Rate   & {$E_{L^2(\Gamma)}: v$} & Rate & {$E_{H^1(\Gamma)}: v$} & Rate\\
\midrule
$ 31 \times 31 \times 31$      & 9.3119E-05 & \textemdash & 1.2573E-04 & \textemdash & 2.9557E-04 & \textemdash\\
$ 63 \times 63 \times 63$      & 2.3982E-05 & 1.96        & 3.2923E-05 & 1.93        & 7.3660E-05 & 2.00\\
$ 127 \times 127 \times 127$   & 6.3155E-06 & 1.93        & 8.5321E-06 & 1.95        & 1.9781E-05 & 1.90\\
$ 255 \times 255 \times 255$   & 1.5774E-06 & 2.00        & 2.1147E-06 & 2.01        & 4.9585E-06 & 2.00\\
\midrule
{$N\times N\times N$} & {$E_{\infty(\Omega)}:\nabla_xu$} & Rate   & {$E_{\infty(\Omega)}:\nabla_y u$} & Rate & {$E_{\infty(\Omega)}:\nabla_z u$} & Rate\\
\midrule
$ 31 \times 31 \times 31$      & 8.2067E-03 & \textemdash & 8.2067E-03 & \textemdash & 2.5934E-03 & \textemdash\\
$ 63 \times 63 \times 63$      & 1.1452E-03 & 2.84        & 1.1452E-03 & 2.84        & 3.6088E-04 & 2.85\\
$ 127 \times 127 \times 127$   & 2.9309E-04 & 1.97        & 2.9309E-04 & 1.97        & 6.3228E-05 & 2.51\\
$ 255 \times 255 \times 255$   & 7.6635E-05 & 1.94        & 7.6635E-05 & 1.94        & 1.5690E-05 & 2.01\\
\bottomrule
\end{tabular}
\caption{Convergence of the $\infty$-, $L^2$- and $H^1$-norm errors of the solutions in the bulk/surface, and the ${\infty}$-norm errors of gradients in the bulk for the model \eqref{eqn:bulk}--\eqref{eqn:surface-ic} with linear bulk-surface coupling. The exact solutions are $u=e^te^{-x(x-1)-y(y-1)}$ and $v = e^te^{-x(x-1)-y(y-1)}(1+x(1-2x)+y(1-2y))$ until final time $T=0.1$ in the sphere of $R=1$. The number of spherical harmonics for terms $v$ and $u_{rr}$ is 529 per each term.}
\label{table:linear-coupling}
\end{table}

In Table~\ref{table:linear-coupling}, we observe second-order accuracy for all norms of the solutions in the bulk and on the surface, together with the second-order accuracy in the components of the gradients. The relative larger errors of $L^2$-norm on the surface, compared to the $\infty$-norm, again can be similarly explained by the inequalities \eqref{eqn:inequality1}--\eqref{eqn:inequality4}.

\begin{figure}
    \centering
    \begin{subfigure}[b]{0.49\textwidth}
        \includegraphics[width=\textwidth]{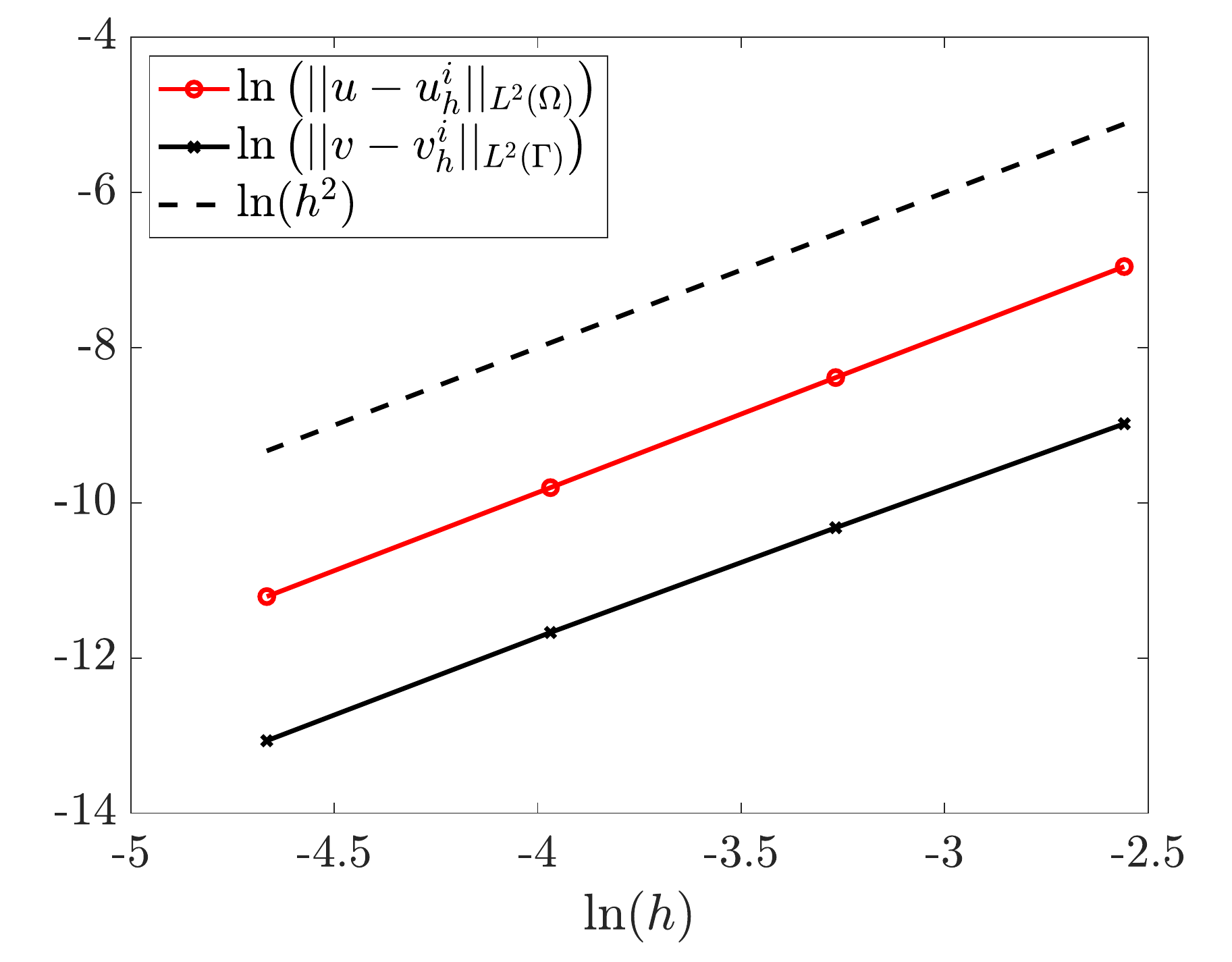}
        \caption{}
    \end{subfigure}
    \begin{subfigure}[b]{0.49\textwidth}
        \includegraphics[width=\textwidth]{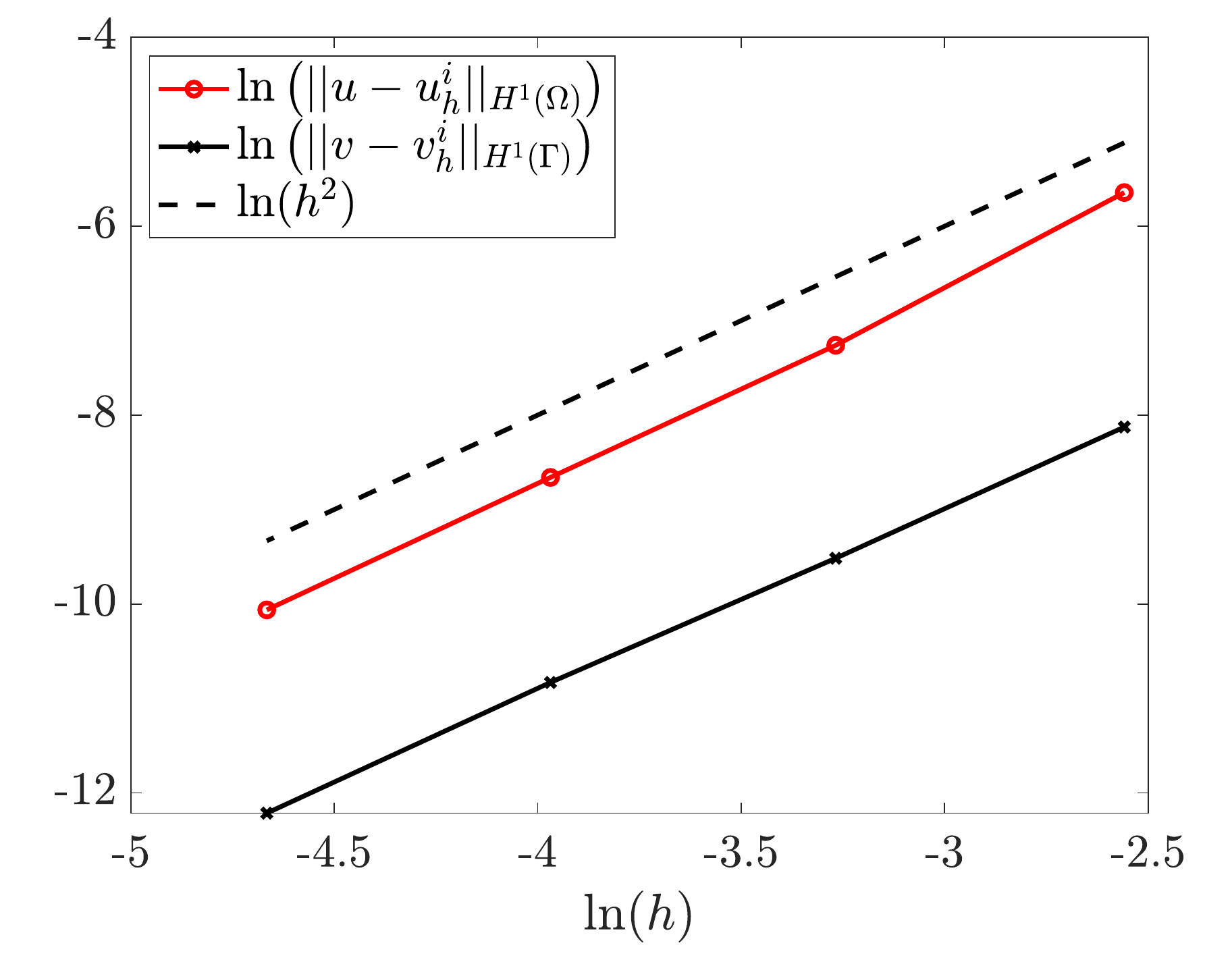}
        \caption{}
    \end{subfigure}
    \caption{Log-log plots of bulk/surface $L^2$-norm errors (left figure) and bulk/surface $H^1$-norm  errors (right figure) for the model \eqref{eqn:bulk}--\eqref{eqn:surface-ic} with linear bulk-surface coupling $h(u,v)=u-v$. The exact solutions are $u=e^te^{-x(x-1)-y(y-1)}$ and $v = e^te^{-x(x-1)-y(y-1)}(1+x(1-2x)+y(1-2y))$ with final time $T=0.1$ in the sphere of $R=1$.}
    \label{fig:linear-bulk-surf-errors}
\end{figure}

In Fig.~\ref{fig:linear-bulk-surf-errors}, we observe second-order convergence for both the $L^2$- and $H^1$-norm errors in the bulk and on the surface. In contrast, the bulk/surface $H^1$-norm errors in the cut finite element approach \cite[Fig. 4]{Burman_2015} are only first order accurate. Furthermore, compared to \cite[Fig. 4]{Burman_2015}, the approach based on DPM in this work gives much smaller $L^2$-norm errors both in the bulk and on the surface.

In the meantime, we notice that in Fig.~\ref{fig:linear-bulk-surf-errors}, the errors on the surface are smaller than the errors in the bulk, which is different from the results of the models with dynamic boundary conditions, see Figs.~\ref{fig:dynamicBC-T1-errors} and \ref{fig:dynamicBC-T2-errors}. Nevertheless, the second-order convergence rates are recovered in all cases.

\begin{figure}
    \centering
    \begin{subfigure}[b]{0.7\textwidth}
        \includegraphics[width=\textwidth]{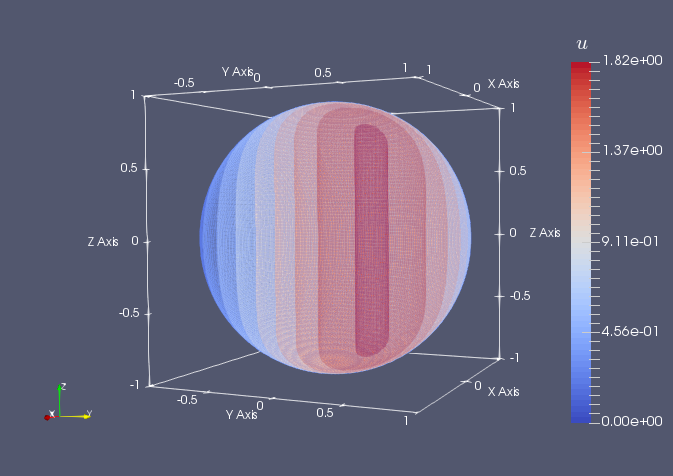}
        \caption{}
    \end{subfigure}
    \begin{subfigure}[b]{0.7\textwidth}
        \includegraphics[width=\textwidth]{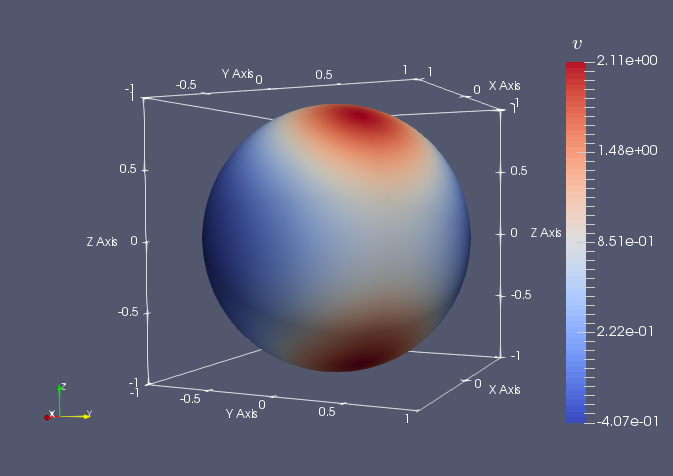}
        \caption{}
    \end{subfigure}
    \caption{3D views of the bulk (top figure) and the surface (bottom figure) approximations from mesh $255\times255\times255$ at $T=0.1$ to the model \eqref{eqn:bulk}--\eqref{eqn:surface-ic} with linear bulk-surface coupling $h(u,v)=u-v$. The exact solutions are $u=e^te^{-x(x-1)-y(y-1)}, v = e^te^{-x(x-1)-y(y-1)}(1+x(1-2x)+y(1-2y))$ in the sphere of $R=1$. }
    \label{fig:linear-bulk-surf-3d-view}
\end{figure}

In Fig. \ref{fig:linear-bulk-surf-3d-view}, we illustrate the solution via the 3D isosurface plots in the top figure and the plot of the surface solution in the bottom figure, obtained on mesh $255\times255\times255$ at final time $T=0.1$. The bottom figure in Fig. \ref{fig:linear-bulk-surf-3d-view} can also be compared to \cite[Fig. 3]{Burman_2015}. In this work, we are able to recover a better resolution of the solution on the surface using the DPM-based algorithms.

%%%%%%%%%%%%%%%%%%%%%%%%%%%%%%%%%%%%%%%%%%%%%%%%%%%%%%%%%%%%%%%%%%%%%%%%%%%%%%%%%%%%%%%%%%%%%%%%%%%%%%%%%%%%
\subsection{Nonlinear Bulk-Surface Coupling}
In this subsection, we demonstrate the numerical results for the
models \eqref{eqn:bulk}--\eqref{eqn:surface-ic} with nonlinear
bulk-surface coupling $h(u,v)=uv$ in the spherical domain of radius
$R=1$. The considered model is motivated by the examples
of the nonlinear bulk-surface coupling from \cite{Elliott_2017,Hansbo_2016}.

\subsubsection{Test 1 for Nonlinear Coupling}

As a first test here, we consider the exact solutions $u(x,y,z,t)=e^te^{-x(x-1)-y(y-1)}$ in the bulk and $v(x,y,z,t) = e^te^{-x(x-1)-y(y-1)}(1+x(1-2x)+y(1-2y))$ on the surface. The motivation to use the same exact solutions as in the linear coupling is that, we can compare the performance of the algorithm for linear/nonlinear bulk-surface coupling and test the robustness of the numerical algorithm based on DPM.

Note that, we do not have exact nonlinear coupling as in \eqref{eqn:coupling} if we use the above exact solutions. Instead, we need to supply a source function $w$ in the coupling, i.e.,
\begin{align}\label{eqn:coupling-with-forcing}
-n\cdot\nabla u = uv + \mathrm{w},\quad (x,y,z,t)\in\Gamma\times\mathbb{R}^+.
\end{align}
Here, the source function $\mathrm{w}$ is computed from the exact solutions $u$ and $v$. The discrete version of \eqref{eqn:coupling-with-forcing} is 
\begin{align}
-n\cdot\nabla u^{i+1} = u^{i+1}v^{i+1}+\mathrm{w}^{i+1}
\end{align}
which can be linearized, if the term $v^{i+1}$ is approximated by either the 2-term approximation \eqref{eqn:2-term-approx}, or the 3-term approximation \eqref{eqn:3-term-approx}.

\begin{table}
\centering
\footnotesize
\sisetup{
  output-exponent-marker = \text{ E},
  exponent-product={},
  retain-explicit-plus
}
\begin{tabular}{c S[table-format=1.4e2] c S[table-format=1.4e2] c S[table-format=1.4e2] c }
\toprule
{$N\times N\times N$} & {$E_{\infty(\Omega)}: u $} & Rate   & {$E_{L^2(\Omega)}: u$} & Rate & {$E_{H^1(\Omega)}: u$} & Rate\\
\midrule
$ 31 \times 31 \times 31$      & 2.1745E-03 & \textemdash & 1.4171E-03 & \textemdash & 4.5547E-03 & \textemdash\\
$ 63 \times 63 \times 63$      & 6.6223E-04 & 1.72        & 4.0246E-04 & 1.82        & 1.1225E-03 & 2.02\\
$ 127 \times 127 \times 127$   & 1.8343E-04 & 1.85        & 1.0690E-04 & 1.91        & 3.1453E-04 & 1.84\\\
$ 255 \times 255 \times 255$   & 4.6212E-05 & 1.99        & 2.7124E-05 & 1.98        & 7.8796E-05 & 2.00\\
\midrule
{$N\times N\times N$} & {$E_{\infty(\Gamma)}: v$} & Rate   & {$E_{L^2(\Gamma)}: v$} & Rate & {$E_{H^1(\Gamma)}: v$} & Rate\\
\midrule
$ 31 \times 31 \times 31$      & 1.2462E-04 & \textemdash & 1.7529E-04 & \textemdash & 3.6232E-04 & \textemdash\\
$ 63 \times 63 \times 63$      & 5.6149E-05 & 1.15        & 7.6767E-05 & 1.19        & 1.5940E-04 & 1.18\\
$ 127 \times 127 \times 127$   & 1.7791E-05 & 1.66        & 2.3819E-05 & 1.69        & 4.9687E-05 & 1.68\\
$ 255 \times 255 \times 255$   & 4.6461E-06 & 1.94        & 6.3801E-06 & 1.90        & 1.2982E-05 & 1.94\\
\midrule
{$N\times N\times N$} & {$E_{\infty(\Omega)}:\nabla_xu$} & Rate   & {$E_{\infty(\Omega)}:\nabla_y u$} & Rate & {$E_{\infty(\Omega)}:\nabla_z u$} & Rate\\
\midrule
$ 31 \times 31 \times 31$      & 8.5460E-03 & \textemdash & 8.5460E-03 & \textemdash & 4.2336E-03 & \textemdash\\
$ 63 \times 63 \times 63$      & 1.0922E-03 & 2.97        & 1.0922E-03 & 2.97        & 1.3177E-03 & 1.68\\
$ 127 \times 127 \times 127$   & 3.1575E-04 & 1.79        & 3.1575E-04 & 1.79        & 3.6751E-04 & 1.84\\
$ 255 \times 255 \times 255$   & 7.5933E-05 & 2.06        & 7.5933E-05 & 2.06        & 9.9683E-05 & 1.88\\
\bottomrule
\end{tabular}
\caption{Convergence of the $\infty$-, $L^2$- and $H^1$-norm errors of the solutions in the bulk/surface, and the ${\infty}$-norm errors of gradients in the bulk for the model \eqref{eqn:bulk}--\eqref{eqn:surface-ic} with nonlinear bulk-surface coupling. The exact solutions are $u=e^te^{-x(x-1)-y(y-1)}$ and $v = e^te^{-x(x-1)-y(y-1)}(1+x(1-2x)+y(1-2y))$ until final time $T=0.1$ in the sphere of $R=1$. The number of spherical harmonics for terms $u$, $v$ and $u_{rr}$ is 529 per each term, and $v^{i+1}\approx v^i+\Delta t v^i_t$.}
\label{table:nonlinear-coupling-Vt-T1}
\end{table}

\begin{table}
\centering
\footnotesize
\sisetup{
  output-exponent-marker = \text{ E},
  exponent-product={},
  retain-explicit-plus
}
\begin{tabular}{c S[table-format=1.4e2] c S[table-format=1.4e2] c S[table-format=1.4e2] c }
\toprule
{$N\times N\times N$} & {$E_{\infty(\Omega)}: u $} & Rate   & {$E_{L^2(\Omega)}: u$} & Rate & {$E_{H^1(\Omega)}: u$} & Rate\\
\midrule
$ 31 \times 31 \times 31$      & 1.2442E-03 & \textemdash & 9.5589E-04 & \textemdash & 3.6512E-03 & \textemdash\\
$ 63 \times 63 \times 63$      & 3.0007E-04 & 2.05        & 2.2875E-04 & 2.06        & 7.1388E-04 & 2.35\\
$ 127 \times 127 \times 127$   & 7.2477E-05 & 2.05        & 5.4911E-05 & 2.06        & 1.7713E-04 & 2.01\\
$ 255 \times 255 \times 255$   & 1.7687E-05 & 2.03        & 1.3390E-05 & 2.04        & 4.3475E-05 & 2.03\\
\midrule
{$N\times N\times N$} & {$E_{\infty(\Gamma)}: v$} & Rate   & {$E_{L^2(\Gamma)}: v$} & Rate & {$E_{H^1(\Gamma)}: v$} & Rate\\
\midrule
$ 31 \times 31 \times 31$      & 1.1314E-04 & \textemdash & 1.3360E-04 & \textemdash & 2.9799E-04 & \textemdash\\
$ 63 \times 63 \times 63$      & 2.9023E-05 & 1.96        & 3.4459E-05 & 1.96        & 7.6480E-05 & 1.96\\
$ 127 \times 127 \times 127$   & 7.7800E-06 & 1.90        & 9.2502E-06 & 1.90        & 2.0684E-05 & 1.89\\
$ 255 \times 255 \times 255$   & 1.9908E-06 & 1.97        & 2.3611E-06 & 1.97        & 5.2992E-06 & 1.96\\
\midrule
{$N\times N\times N$} & {$E_{\infty(\Omega)}:\nabla_xu$} & Rate   & {$E_{\infty(\Omega)}:\nabla_y u$} & Rate & {$E_{\infty(\Omega)}:\nabla_z u$} & Rate\\
\midrule
$ 31 \times 31 \times 31$      & 8.5799E-03 & \textemdash & 8.5799E-03 & \textemdash & 2.4914E-03 & \textemdash\\
$ 63 \times 63 \times 63$      & 1.1211E-03 & 2.94        & 1.1211E-03 & 2.94        & 4.0347E-04 & 2.63\\
$ 127 \times 127 \times 127$   & 2.9313E-04 & 1.94        & 2.9313E-04 & 1.94        & 1.0618E-04 & 1.93\\
$ 255 \times 255 \times 255$   & 7.6440E-05 & 1.94        & 7.6440E-05 & 1.94        & 2.6384E-05 & 2.01\\
\bottomrule
\end{tabular}
\caption{Convergence of the $\infty$-, $L^2$- and $H^1$-norm errors of the solutions in the bulk/surface, and the ${\infty}$-norm errors of gradients in the bulk for the model \eqref{eqn:bulk}--\eqref{eqn:surface-ic} with nonlinear bulk-surface coupling. The exact solutions are $u=e^te^{-x(x-1)-y(y-1)}$ and $v = e^te^{-x(x-1)-y(y-1)}(1+x(1-2x)+y(1-2y))$ until final time $T=0.1$ in the sphere of $R=1$. The number of spherical harmonics for terms $u$, $v$ and $u_{rr}$ is 529 per each term, and $v^{i+1}\approx v^i+\Delta t v^i_t+\Delta t^2 v^i_{tt}/2$.}
\label{table:nonlinear-coupling-Vtt-T1}
\end{table}

The errors in Table~\ref{table:nonlinear-coupling-Vt-T1} correspond to 2-term approximation \eqref{eqn:2-term-approx} for $v^{i+1}$. With the 2-term approximation \eqref{eqn:2-term-approx}, the ${\infty}$-, $L^2$-, $H^1$-norm errors of solutions in the bulk and ${\infty}$-norm errors of the gradients in the bulk all obey optimal second-order convergence. Meanwhile, the ${\infty}$-, $L^2$-, $H^1$-norm errors of the solution on the surface give sub-optimal second-order accuracy in the first few coarser meshes. However, the second-order accuracy is recovered on finer meshes, e.g., on mesh $255\times255\times255$.

In Table \ref{table:nonlinear-coupling-Vtt-T1}, we adopt the 3-term approximation \eqref{eqn:3-term-approx} for the $v^{i+1}$ term. There are slight improvements of the accuracy for the solutions and gradients in the bulk. In the meantime, the accuracy of the solution on the surface is improved and second-order accuracy is recovered even on coarser meshes. Note that, there are barely any added computational cost, when one switches from using 2-term approximation \eqref{eqn:2-term-approx} to 3-term approximation \eqref{eqn:3-term-approx} for the $v^{i+1}$ term.

Moreover, the errors in the bulk and on the surface for the nonlinear
bulk-surface coupling in Table \ref{table:nonlinear-coupling-Vtt-T1}
are very similar to the errors for linear bulk-surface coupling in
Table \ref{table:linear-coupling}. This illustrates the robustness of
the designed DPM-based algorithm. Also note that, the algorithm for
nonlinear coupling is very similar to the ones for linear
coupling. The only difference is that the matrix $C'$ in the least
square system \eqref{eqn:LS-system-2} needs to be updated and the resulting normal matrices need to be inverted
at each time step, which makes it more expensive. Hence, it is
advantageous to use the reduced BEPs as it is done in the current work.

\begin{figure}
    \centering
    \begin{subfigure}[b]{0.49\textwidth}
        \includegraphics[width=\textwidth]{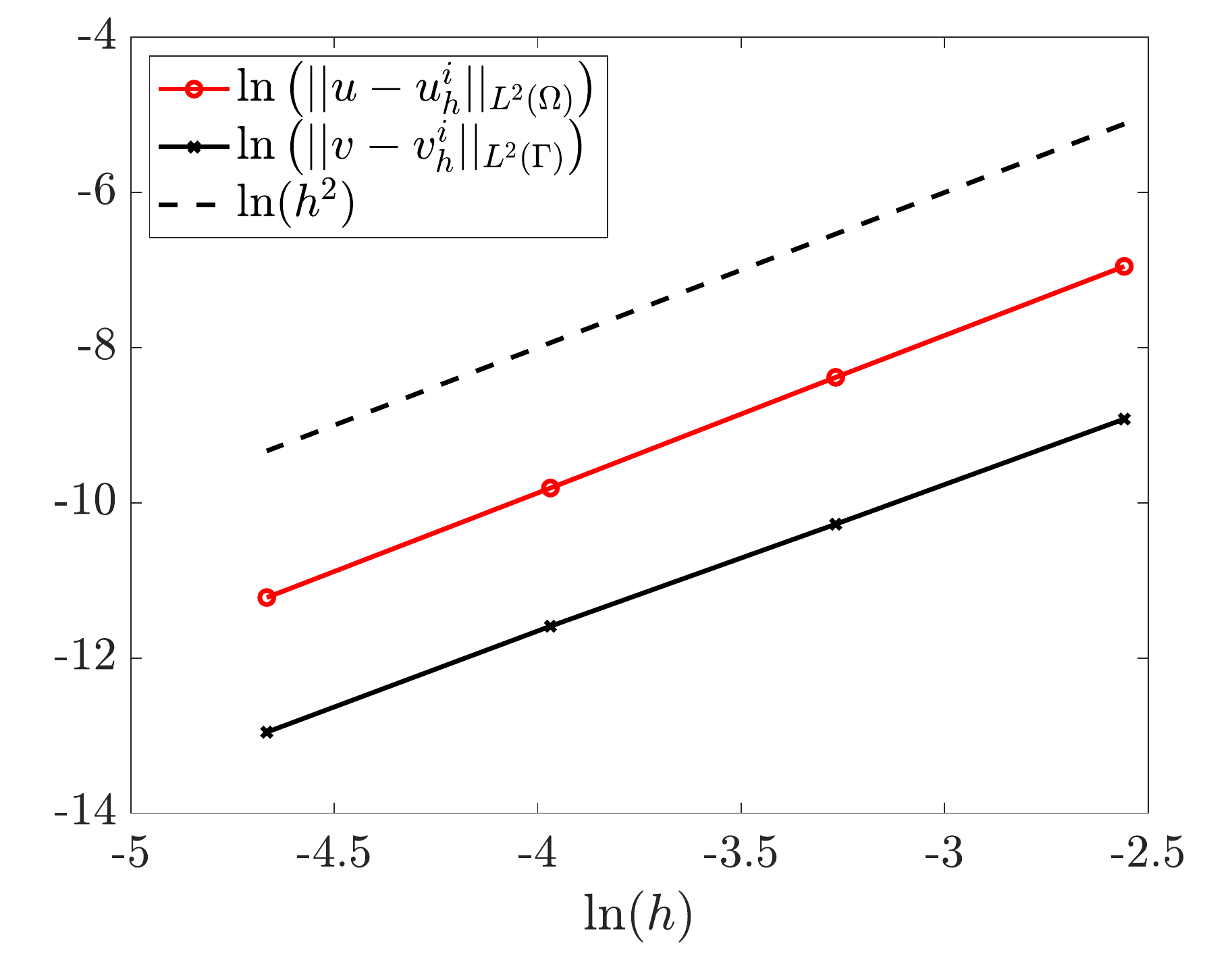}
        \caption{}
    \end{subfigure}
    \begin{subfigure}[b]{0.49\textwidth}
        \includegraphics[width=\textwidth]{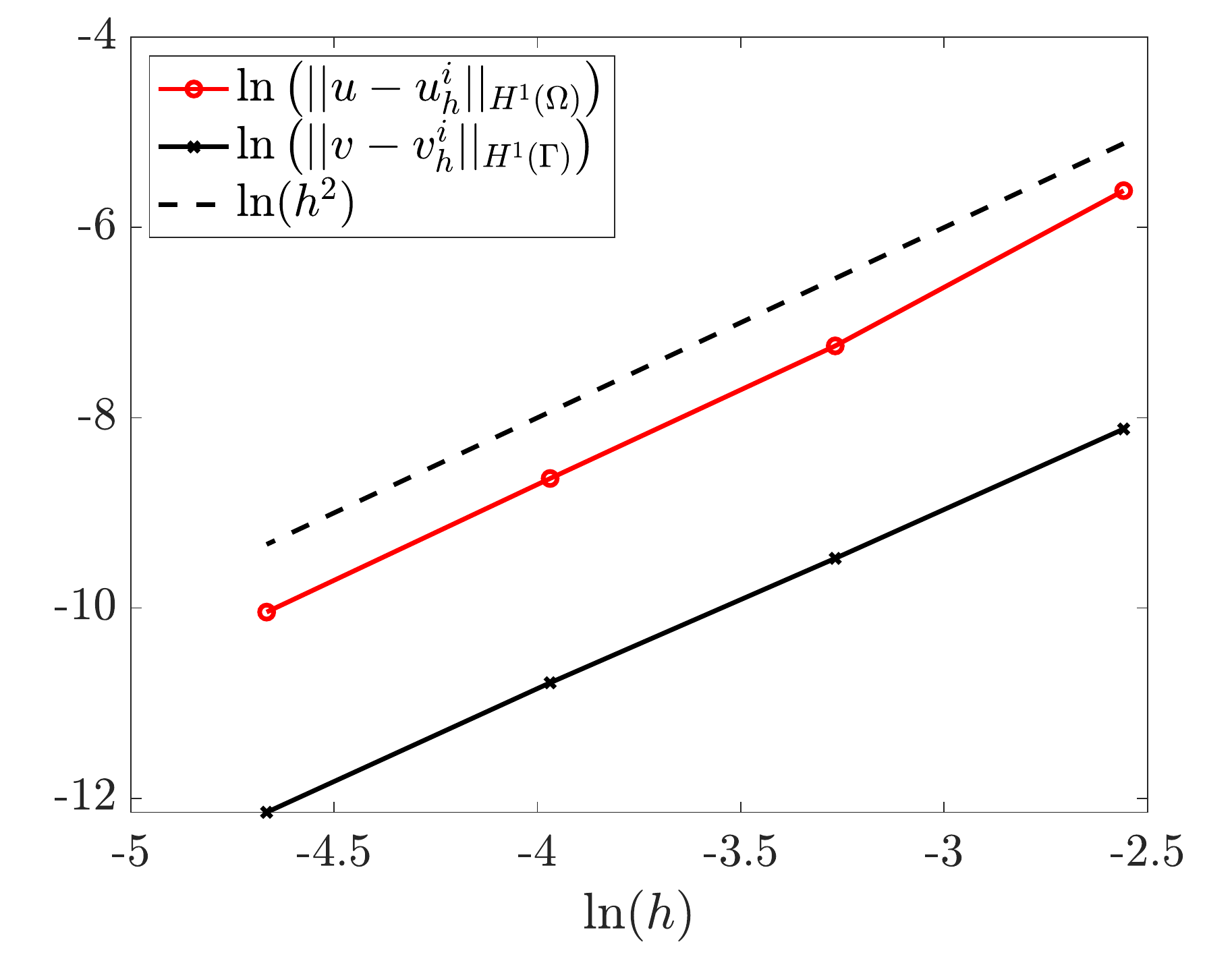}
        \caption{}
    \end{subfigure}
    \caption{Log-log plots of bulk/surface $L^2$-norm errors (left figure) and bulk/surface $H^1$-norm errors (right figure) for the model \eqref{eqn:bulk}--\eqref{eqn:surface-ic} with nonlinear bulk-surface coupling $h(u,v)=uv$. The exact solutions are $u=e^te^{-x(x-1)-y(y-1)}$ and $v = e^te^{-x(x-1)-y(y-1)}(1+x(1-2x)+y(1-2y))$ until final time $T=0.1$ in the sphere of $R=1$, and $v^{i+1}\approx v^i+\Delta t v^i_t+\Delta t^2 v^i_{tt}/2$.}
    \label{fig:nonlinear-coupling-T1-errors}
\end{figure}

Again, the plots of $L^2$- and $H^1$-norm errors of the nonlinear coupling in Fig. \ref{fig:nonlinear-coupling-T1-errors} are similar to the plots of errors in the linear coupling, see Fig.~\ref{fig:linear-bulk-surf-errors}. In Fig. \ref{fig:nonlinear-coupling-T1-3d-view}, there is no observable difference in the isosurface plots in the bulk and the surface plots from the plots for the linear-coupling case, obtained on mesh $255\times255\times255$ at final time $T=0.1$, see Fig.~\ref{fig:linear-bulk-surf-3d-view}.

\begin{figure}
    \centering
    \begin{subfigure}[b]{0.7\textwidth}
        \includegraphics[width=\textwidth]{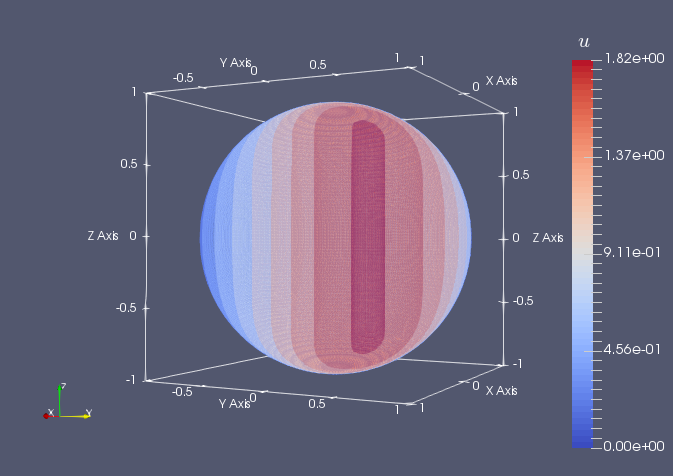}
        \caption{}
    \end{subfigure}
    \begin{subfigure}[b]{0.7\textwidth}
        \includegraphics[width=\textwidth]{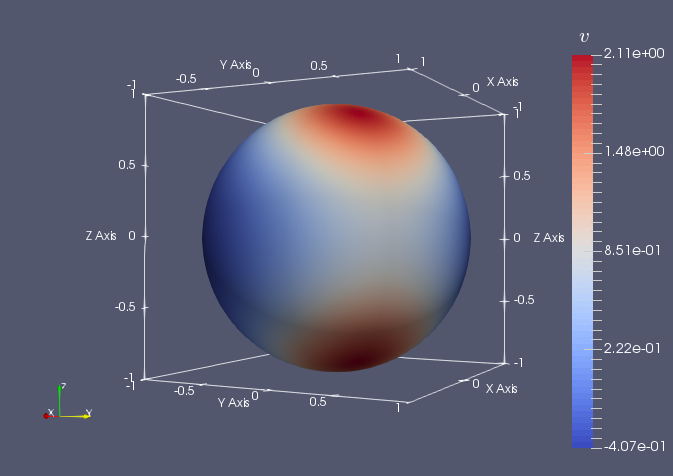}
        \caption{}
    \end{subfigure}
    \caption{3D views of the bulk (top figure) and surface (bottom figure) approximations from mesh $255\times255\times255$ at $T=0.1$ to the model \eqref{eqn:bulk}--\eqref{eqn:surface-ic} of nonlinear bulk-surface coupling $h(u,v)=uv$. The exact solutions are $u=e^te^{-x(x-1)-y(y-1)}, v = e^te^{-x(x-1)-y(y-1)}(1+x(1-2x)+y(1-2y))$in the sphere of $R=1$, and $v^{i+1}\approx v^i+\Delta t v^i_t+\Delta t^2 v^i_{tt}/2$.}
    \label{fig:nonlinear-coupling-T1-3d-view}
\end{figure}

\subsubsection{Test 2 for Nonlinear Coupling}

In this subsection, we employ the exact solutions $u=v=e^t\sin(x)\sin(2y)\sin(3z)$ both in the bulk and on the surface, as the ones we use in the second test of the models with dynamic boundary conditions. Again, second order accuracy are observed in Tables~\ref{table:nonlinear-coupling-Vt-T2} and \ref{table:nonlinear-coupling-Vtt-T2} for the ${\infty}$-, $L^2$- and $H^1$-norm errors. It is also interesting to notice that for this pair of exact solutions, 2-term approximation \eqref{eqn:2-term-approx} and 3-term approximation \eqref{eqn:3-term-approx} of the $v^{i+1}$ term give very similar convergence results, which again illustrates the robustness of the proposed DPM-based algorithms.

\begin{table}
\centering
\footnotesize
\sisetup{
  output-exponent-marker = \text{ E},
  exponent-product={},
  retain-explicit-plus
}
\begin{tabular}{c S[table-format=1.4e2] c S[table-format=1.4e2] c S[table-format=1.4e2] c }
\toprule
{$N\times N\times N$} & {$E_{\infty(\Omega)}: u $} & Rate   & {$E_{L^2(\Omega)}: u$} & Rate & {$E_{H^1(\Omega)}: u$} & Rate\\
\midrule
$ 31 \times 31 \times 31$      & 1.3914E-03 & \textemdash & 1.0543E-03 & \textemdash & 5.0230E-03 & \textemdash\\
$ 63 \times 63 \times 63$      & 3.6065E-04 & 1.95        & 2.6657E-04 & 1.98        & 1.2706E-03 & 1.98\\
$ 127 \times 127 \times 127$   & 9.4354E-05 & 1.93        & 6.6522E-05 & 2.00        & 3.1710E-04 & 2.00\\\
$ 255 \times 255 \times 255$   & 2.3408E-05 & 2.01        & 1.6608E-05 & 2.00        & 7.9165E-05 & 2.00\\
\midrule
{$N\times N\times N$} & {$E_{\infty(\Gamma)}: v$} & Rate   & {$E_{L^2(\Gamma)}: v$} & Rate & {$E_{H^1(\Gamma)}: v$} & Rate\\
\midrule
$ 31 \times 31 \times 31$      & 2.1011E-05 & \textemdash & 2.7089E-05 & \textemdash & 9.9342E-05 & \textemdash\\
$ 63 \times 63 \times 63$      & 6.5900E-06 & 1.67        & 7.7509E-06 & 1.81        & 2.6332E-05 & 1.92\\
$ 127 \times 127 \times 127$   & 1.8191E-06 & 1.86        & 2.3303E-06 & 1.73        & 7.1906E-06 & 1.87\\
$ 255 \times 255 \times 255$   & 4.6581E-07 & 1.97        & 6.0249E-07 & 1.95        & 1.8064E-06 & 1.99\\
\midrule
{$N\times N\times N$} & {$E_{\infty(\Omega)}:\nabla_xu$} & Rate   & {$E_{\infty(\Omega)}:\nabla_y u$} & Rate & {$E_{\infty(\Omega)}:\nabla_z u$} & Rate\\
\midrule
$ 31 \times 31 \times 31$      & 3.3321E-03 & \textemdash & 4.5730E-03 & \textemdash & 4.9888E-03 & \textemdash\\
$ 63 \times 63 \times 63$      & 8.3625E-04 & 1.99        & 1.0630E-03 & 2.11        & 1.1784E-03 & 2.08\\
$ 127 \times 127 \times 127$   & 2.0870E-04 & 2.00        & 2.6904E-04 & 1.98        & 2.8762E-04 & 2.03\\
$ 255 \times 255 \times 255$   & 5.2251E-05 & 2.00        & 6.7303E-05 & 2.00        & 7.3246E-05 & 1.97\\
\bottomrule
\end{tabular}
\caption{Convergence of the $\infty$-, $L^2$- and $H^1$-norm errors of the solutions in the bulk/surface, and the ${\infty}$-norm errors of gradients in the bulk for the model \eqref{eqn:bulk}--\eqref{eqn:surface-ic} with nonlinear bulk-surface coupling. The exact solutions are $u=v=e^t\sin(x)\sin(2y)\sin(3z)$ until final time $T=0.1$ in the sphere of $R=1$. The number of spherical harmonics for terms $u$, $v$ and $u_{rr}$ is 400 per each term and $v^{i+1}\approx v^i+\Delta t v^i_t$.}
\label{table:nonlinear-coupling-Vt-T2}
\end{table}

\begin{table}
\centering
\footnotesize
\sisetup{
  output-exponent-marker = \text{ E},
  exponent-product={},
  retain-explicit-plus
}
\begin{tabular}{c S[table-format=1.4e2] c S[table-format=1.4e2] c S[table-format=1.4e2] c }
\toprule
{$N\times N\times N$} & {$E_{\infty(\Omega)}: u $} & Rate   & {$E_{L^2(\Omega)}: u$} & Rate & {$E_{H^1(\Omega)}: u$} & Rate\\
\midrule
$ 31 \times 31 \times 31$      & 1.5302E-03 & \textemdash & 1.0646E-03 & \textemdash & 5.0709E-03 & \textemdash\\
$ 63 \times 63 \times 63$      & 4.1025E-04 & 1.90        & 2.6970E-04 & 1.98        & 1.2861E-03 & 1.98\\
$ 127 \times 127 \times 127$   & 1.0691E-04 & 1.94        & 6.7408E-05 & 2.00        & 3.2158E-04 & 2.00\\
$ 255 \times 255 \times 255$   & 2.6478E-05 & 2.01        & 1.6839E-05 & 2.00        & 8.0352E-05 & 2.00\\
\midrule
{$N\times N\times N$} & {$E_{\infty(\Gamma)}: v$} & Rate   & {$E_{L^2(\Gamma)}: v$} & Rate & {$E_{H^1(\Gamma)}: v$} & Rate\\
\midrule
$ 31 \times 31 \times 31$      & 3.4599E-05 & \textemdash & 4.8766E-05 & \textemdash & 1.2807E-04 & \textemdash\\
$ 63 \times 63 \times 63$      & 9.7036E-06 & 1.83        & 1.3621E-05 & 1.84        & 3.5923E-05 & 1.83\\
$ 127 \times 127 \times 127$   & 2.6424E-06 & 1.88        & 3.6432E-06 & 1.90        & 9.6663E-06 & 1.89\\
$ 255 \times 255 \times 255$   & 6.7367E-07 & 1.97        & 9.3640E-07 & 1.96        & 2.4564E-06 & 1.98\\
\midrule
{$N\times N\times N$} & {$E_{\infty(\Omega)}:\nabla_xu$} & Rate   & {$E_{\infty(\Omega)}:\nabla_y u$} & Rate & {$E_{\infty(\Omega)}:\nabla_z u$} & Rate\\
\midrule
$ 31 \times 31 \times 31$      & 3.4453E-03 & \textemdash & 4.6325E-03 & \textemdash & 5.0779E-03 & \textemdash\\
$ 63 \times 63 \times 63$      & 8.8748E-04 & 1.96        & 1.0970E-03 & 2.08        & 1.1997E-03 & 2.08\\
$ 127 \times 127 \times 127$   & 2.1871E-04 & 2.02        & 2.7758E-04 & 1.98        & 3.0294E-04 & 1.99\\
$ 255 \times 255 \times 255$   & 5.5442E-05 & 1.98        & 6.9999E-05 & 1.99        & 7.5981E-05 & 2.00\\
\bottomrule
\end{tabular}
\caption{Convergence of the $\infty$-, $L^2$- and $H^1$-norm errors of the solutions in the bulk/surface, and the ${\infty}$-norm errors of gradients in the bulk for the model \eqref{eqn:bulk}--\eqref{eqn:surface-ic} with nonlinear bulk-surface coupling. The exact solutions are $u=v=e^t\sin(x)\sin(2y)\sin(3z)$ until final time $T=0.1$ in the sphere of $R=1$. The number of spherical harmonics for terms $u$, $v$ and $u_{rr}$ is 400 per each term and $v^{i+1}\approx v^i+\Delta t v^i_t+\Delta t^2 v^i_{tt}/2$.}
\label{table:nonlinear-coupling-Vtt-T2}
\end{table}

\begin{figure}
    \centering
    \begin{subfigure}[b]{0.49\textwidth}
        \includegraphics[width=\textwidth]{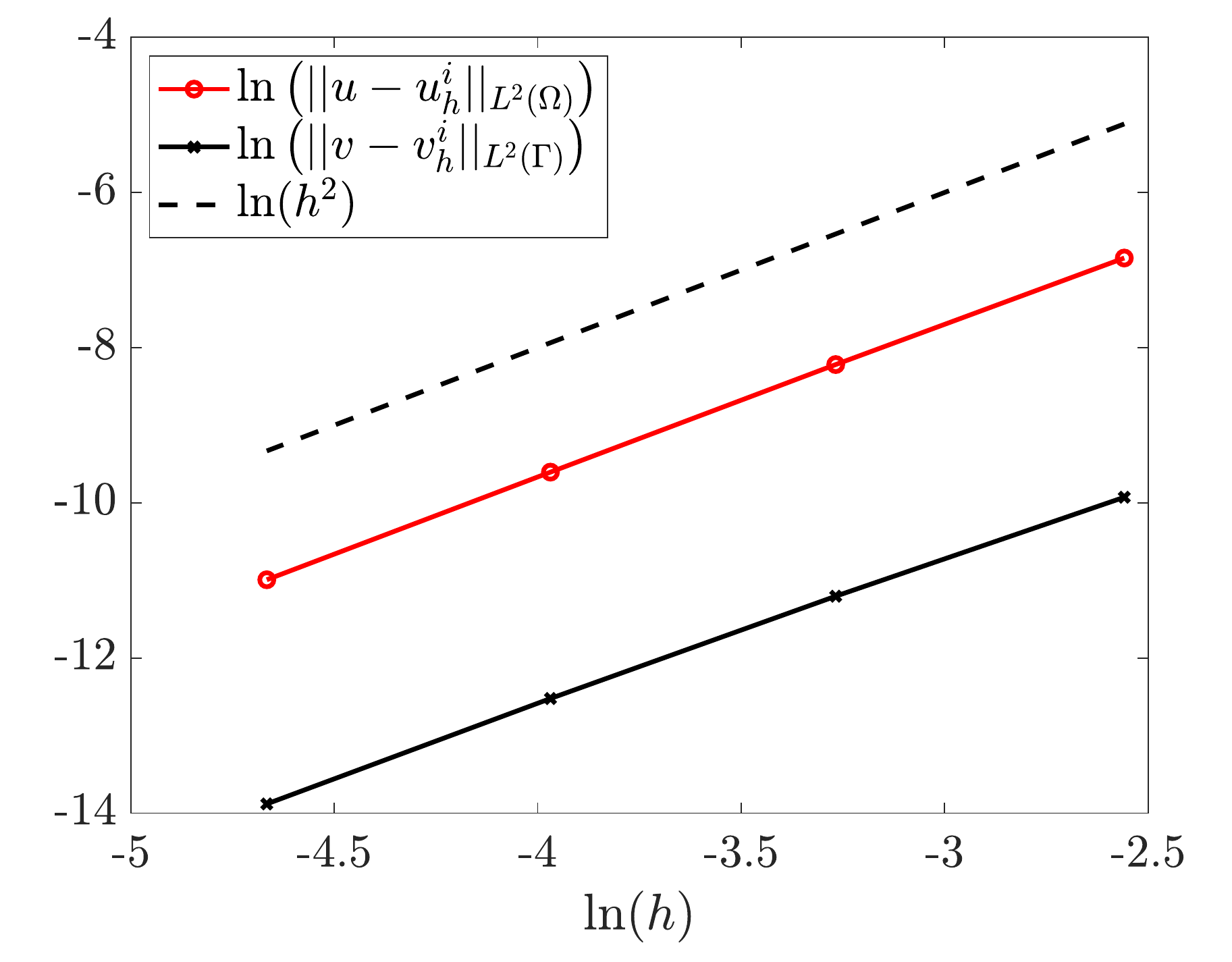}
        \caption{}
    \end{subfigure}
    \begin{subfigure}[b]{0.49\textwidth}
        \includegraphics[width=\textwidth]{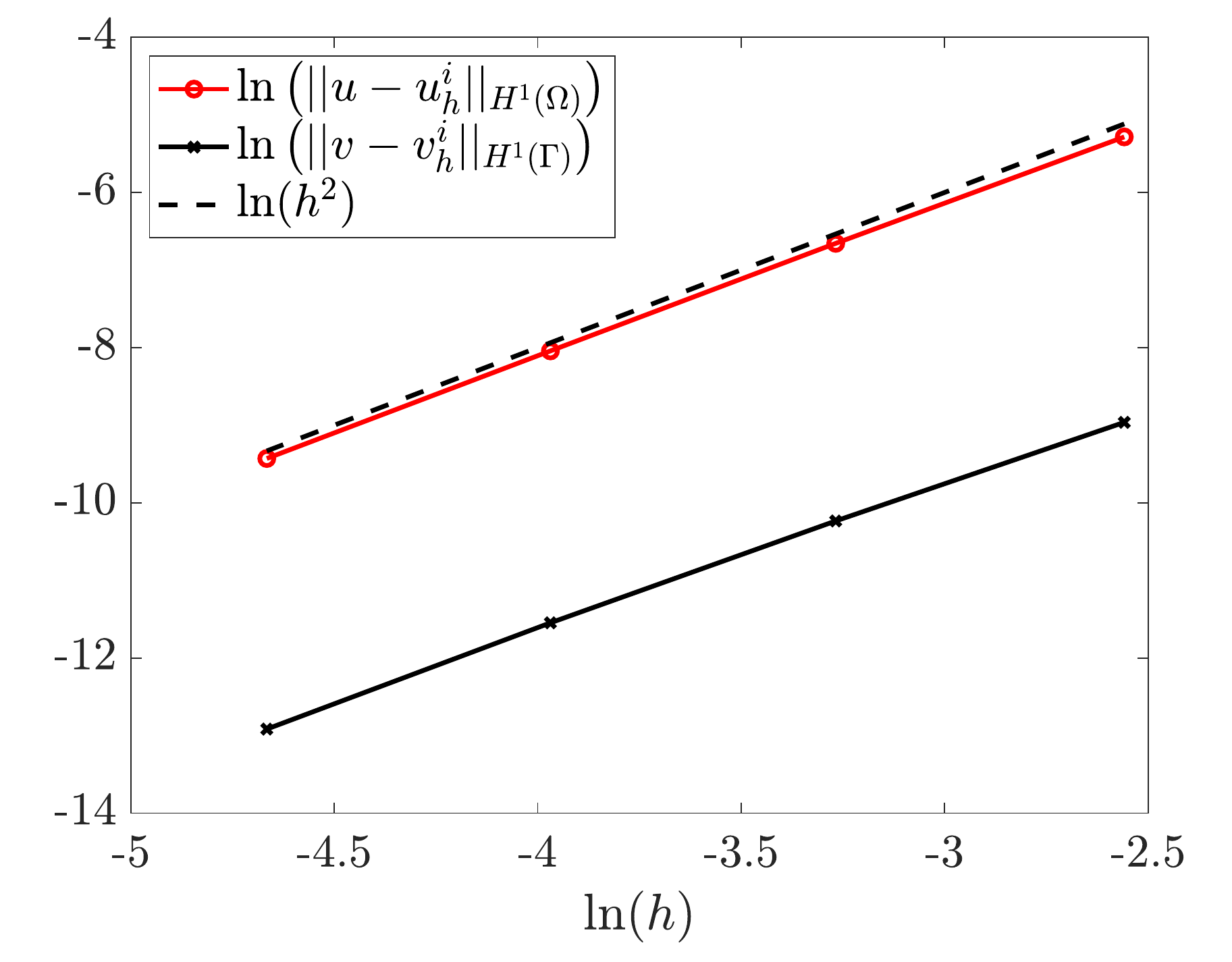}
        \caption{}
    \end{subfigure}
    \caption{Log-log plots of bulk/surface $L^2$-norm errors (left figure) and bulk/surface $H^1$-norm errors (right figure) for the model \eqref{eqn:bulk}--\eqref{eqn:surface-ic} with nonlinear bulk-surface coupling $h(u,v)=uv$. The exact solutions are $u=v= e^t\sin(x)\sin(2y)\sin(3z)$ until final time $T=0.1$ in the sphere of $R=1$, and $v^{i+1}\approx v^i+\Delta t v^i_t+\Delta t^2 v^i_{tt}/2$.}
    \label{fig:nonlinear-coupling-T2-errors}
\end{figure}

In Fig. \ref{fig:nonlinear-coupling-T2-errors}, we observe second order convergence of $L^2$- and $H^1$-norm errors both in the bulk and on the surface. Unlike the numerical results for dynamic boundary condition in Fig.~\ref{fig:dynamicBC-T2-errors}, the $L^2$- and $H^1$-norm errors in the bulk are larger than the errors on the surface in Fig. \ref{fig:nonlinear-coupling-T2-errors}, which is also observed in the first test of the nonlinear coupling in Fig. \ref{fig:nonlinear-coupling-T1-errors}, as well as in the test of linear bulk-surface coupling in Fig. \ref{fig:linear-bulk-surf-errors}.

In Fig. \ref{fig:nonlinear-coupling-T2-3d-view}, we present the 3D views of the isosurface plots in the bulk and the plot of surface solutions, obtained on mesh $255\times255\times255$ at final time $T=0.1$.

\begin{figure}
    \centering
    \begin{subfigure}[b]{0.7\textwidth}
        \includegraphics[width=\textwidth]{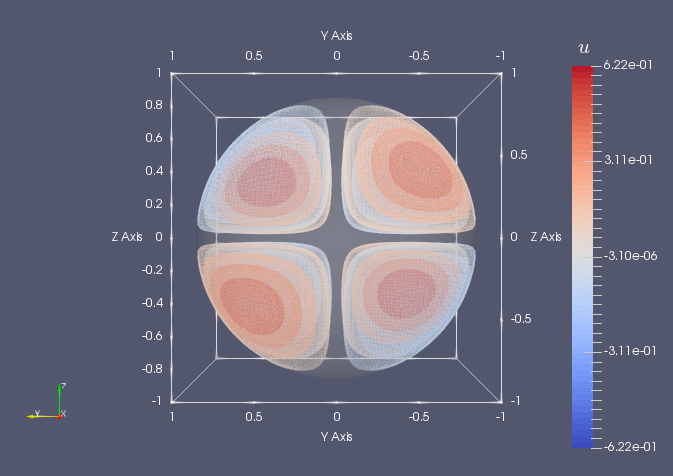}
        \caption{}
    \end{subfigure}
    \begin{subfigure}[b]{0.7\textwidth}
        \includegraphics[width=\textwidth]{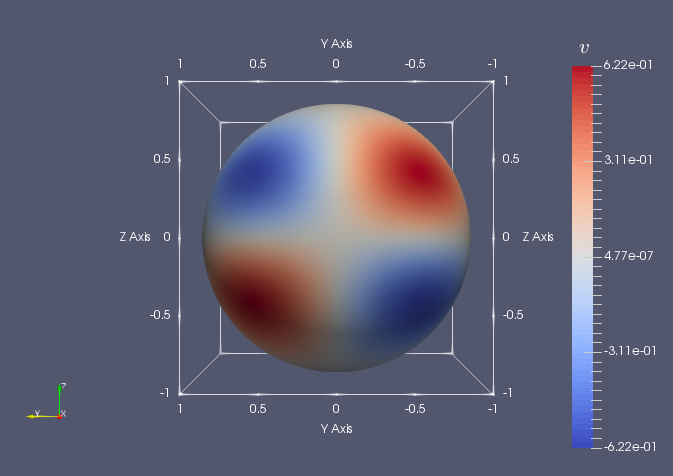}
        \caption{}
    \end{subfigure}
    \caption{3D views of the bulk (top figure) and surface (bottom figure) approximations from mesh $255\times255\times255$ at $T=0.1$ to the model \eqref{eqn:bulk}--\eqref{eqn:surface-ic} of nonlinear bulk-surface coupling $h(u,v)=uv$. The exact solutions are $u=v=e^t\sin(x)\sin(2y)\sin(3z)$ in the sphere of $R=1$, and $v^{i+1}\approx v^i+\Delta t v^i_t+\Delta t^2 v^i_{tt}/2$.}
    \label{fig:nonlinear-coupling-T2-3d-view}
\end{figure}

\subsection{Condition Numbers}
In Table~\ref{table:condition-number}, we demonstrate the condition numbers of the normal matrices from the resulting algebraic systems \eqref{eqn:BEP-reformulated} in \textit{Case 1}, \eqref{eqn:linear-coupling-reformulated-bep} in \textit{Case 2: a)}, and \eqref{eqn:LS-system}--\eqref{eqn:LS-system-2} in \textit{Case 2: b)}. Note that in \textit{Case 1)} and \textit{Case 2: a)}, the normal matrices are pre-computed outside of the time loop, while in \textit{Case 2: b)}, the normal matrices need to be assembled at each time step due to the nonlinearity. Thus, in Table~\ref{table:condition-number}, the condition numbers for \textit{Case 2: b)} are computed only at the first and last time steps. 

\begin{table}
\centering
\footnotesize
\sisetup{
  output-exponent-marker = \text{ E},
  exponent-product={},
  retain-explicit-plus
}
\begin{tabular}{c S[table-format=1.4e2] S[table-format=1.4e2] S[table-format=1.4e2] S[table-format=1.4e2] S[table-format=1.4e2]}
\toprule
{$N$} & {Case 1 Test 2} & {Case 2:a} & {Case 2:b Test 1 } & {Case 2:b Test 1 } & {Case 2:b Test 1}\\
{} & {(at $t=0$)} & {(at $t=0$)} & {(w/o $v_{tt}$ at $t=0$)} & {(w/ $v_{tt}$ at $t=0$)} & {(w/ $v_{tt}$ at $t=0.1$)}\\
\midrule
31  & 3.4838e+01 & 1.7962e+04 & 2.7494e+04 & 2.7565e+04 & 2.2402e+04\\
63  & 4.4475e+01 & 1.7659e+03 & 3.7456e+03 & 3.7483e+03 & 3.8275e+03\\
127 & 3.2709e+01 & 3.4449e+02 & 8.3878e+02 & 8.3901e+02 & 7.5131e+02\\
255 & 4.0182e+01 & 6.9272e+02 & 9.1638e+02 & 9.1639e+02 & 9.0053e+02\\
\bottomrule
\end{tabular}
\caption{Condition number of the normal matrices of BEP (notation ``w/o $v_{tt}$'' denotes 2-term approximation \eqref{eqn:2-term-approx} and notation ``w/ $v_{tt}$'' denotes 3-term approximation \eqref{eqn:3-term-approx}).}
\label{table:condition-number}
\end{table}

Furthermore, we verified that condition numbers of the normal matrices remain in similar magnitude over time for the nonlinear models in \textit{Case 2: b)}. See the last two columns in Table~\ref{table:condition-number}.

%%%%%%%%%%%%%%%%%%%%%%%%%%%%%%%%%%%%%%%%%%%%%%%%%%%%%%%%%%%%%%%%%%%%%%%%%%%%%%%%%%%%%%%%%%%%%%%%%%%%%%%%%%%%%%%%%%%%%%%%%%
\section*{Acknowledgements}
The authors wish to thank P. Bowman and M. Cuma for assistance in computing facility, the CHPC at the University of Utah for providing computing allocations. The authors are also grateful to the referees for their most valuable suggestions. Yekaterina Epshteyn also acknowledges partial support of Simons Foundation Grant No. 415673.

\bibliographystyle{plainnat}
\bibliography{Revision_ArXiv_DynamicBcBulkSurface3D.bib}   % name your BibTeX data base

\end{document}